\documentclass[draft,a4paper]{amsart}

\usepackage{pstricks,pst-node}
\usepackage{amsmath,amsfonts,amssymb,amsthm}
\usepackage{multirow}

\usepackage{url}

\title[Families of Explicit Isogenies of Hyperelliptic Jacobians]{
	Families of Explicit Isogenies \\ of Hyperelliptic Jacobians
}
\author{Benjamin Smith}
\address{
	INRIA Saclay--\^Ile-de-France 
	/ 
	Laboratoire d'Informatique de l'\'Ecole polytechnique (LIX),
	91128 Palaiseau Cedex, France
}
\email{smith@lix.polytechnique.fr}
\theoremstyle{definition}
\newtheorem{definition}{Definition}[section]
\newtheorem{example}{Example}[section]
\theoremstyle{remark}
\newtheorem{remark}{Remark}[section]
\theoremstyle{plain}
\newtheorem{theorem}{Theorem}[section]

\newtheorem{lemma}[theorem]{Lemma}

\newcommand{\CC}{{\mathbb{C}}}
\newcommand{\FF}{{\mathbb{F}}}
\newcommand{\QQ}{{\mathbb{Q}}}
\newcommand{\QQbar}{{\overline{\mathbb{Q}}}}
\newcommand{\ZZ}{{\mathbb{Z}}}
\newcommand{\product}[3][T]{\ensuremath{{#2}\!\times_{#1}\!{#3}}}
\newcommand{\XxY}{\product[]{X}{Y}}
\newcommand{\Jac}[1]{\ensuremath{{J}_{#1}}}

\newcommand{\Jacdual}[1]{\ensuremath{\widehat{J}_{#1}}}
\newcommand{\family}[1]{\ensuremath{\mathcal{#1}}}
\newcommand{\Jacfamily}[1]{\ensuremath{{\family{J}}_{\family{#1}}}}

\newcommand{\Gal}[1]{{\mathrm{Gal}(#1)}}
\newcommand{\Hom}{{\mathrm{Hom}}}
\newcommand{\End}{{\mathrm{End}}}
\newcommand{\Mat}{{\mathrm{Mat}}}
\newcommand{\variety}[1]{\ensuremath{V\!\left({#1}\right)}}

\newcommand{\dualof}[1]{\ensuremath{{#1}^\dagger}}
\newcommand{\compose}[2]{\ensuremath{{#1}\circ{#2}}}
\newcommand{\multiplication}[2][{}]{\ensuremath{[#2]_{#1}}}
\newcommand{\genus}[1]{\ensuremath{{g_{#1}}}}
\newcommand{\differentials}{\ensuremath{\Omega}}
\newcommand{\closure}[1]{\ensuremath{\overline{#1}}}
\newcommand{\Spec}{\mathrm{Spec}}
\newcommand{\hyperellipticmoduli}[1]{\ensuremath{\family{H}_{#1}}}

\newcommand{\abelianmoduli}[1]{\ensuremath{\family{A}_{#1}}}

\newcommand{\isogenytype}[2]{\ensuremath{(\ZZ/{#1}\ZZ)^{#2}}}
\newcommand{\isogenytypetwo}[4]{\ensuremath{(\ZZ/{#1}\ZZ)^{#2}\!\times\!(\ZZ/{#3}\ZZ)^{#4}}}
\newcommand{\isogenytypethree}[6]{\ensuremath{(\ZZ/{#1}\ZZ)^{#2}\!\times\!(\ZZ/{#3}\ZZ)^{#4}\!\times\!(\ZZ/{#3}\ZZ)^{#4}}}
\newcommand{\differentialmatrix}[2][{}]{\ensuremath{D_{#1}({#2})}}
\newcommand{\torsionmatrix}[1]{\ensuremath{T_2({#1})}}
\newcommand{\Tr}{\ensuremath{\mathrm{Tr}}}
\newcommand{\subgrp}[1]{\ensuremath{\left\langle{#1}\right\rangle}}

\begin{document}

\begin{abstract}
	We construct
	three-dimensional families of hyperelliptic curves
	of genus $6$, $12$, and $14$,
	two-dimensional families of hyperelliptic curves
	of genus $3$, $6$, $7$, $10$, $20$, and $30$,
	and one-dimensional families of hyperelliptic curves
	of genus $5$, $10$ and $15$,
	all of which are equipped with an an explicit 
	isogeny from their Jacobian
	to another hyperelliptic Jacobian.
	We show that the Jacobians are generically absolutely simple,
	and describe the kernels of the isogenies.
	The families are derived from Cassou--Nogu\`es and Couveignes'
	explicit classification
	of pairs $(f,g)$ of polynomials such that $f(x_1) - g(x_2)$
	is reducible.
\end{abstract}

\maketitle

\section{%%%%%%%%%%%%%%%%%%%%%%%%%%%%%%%%%%%%%%%%%%%%%%%%%%%%%%%%%%%%%%%%%%%%%%%
	Introduction
}%%%%%%%%%%%%%%%%%%%%%%%%%%%%%%%%%%%%%%%%%%%%%%%%%%%%%%%%%%%%%%%%%%%%%%%%%%%%%%%
\label{sec:introduction}

In this article,
we construct twelve explicit families of isogenies of hyperelliptic Jacobians.
By explicit,
we mean that we provide equations for hyperelliptic curves 
generating the domains and codomains of
each isogeny,
together with a correspondence on the curves 
realizing the isogeny as
a map on divisor classes.
Our main results are summarized by Theorem~\ref{theorem:main},
which follows from 
the examples of \S\ref{sec:examples}.
We also construct some families of 
Jacobians with explicit Real Multiplication,
including and generalizing the families described by
Tautz, Top, and Verberkmoes~\cite{Tautz-Top-Verberkmoes}.

\begin{theorem}
\label{theorem:main}
	For each row of the following table,
	there exists 
	an $n$-dimensional family of explicit isogenies of Jacobians
	of hyperelliptic curves of genus $g$ over $k$,
	with kernel isomorphic to $G$ 
	(and hence splitting multiplication-by-$m$).
	The generic fibre of each family
	is an isogeny of absolutely simple Jacobians.
	\begin{center}
	\begin{tabular}{|r|c|c|l|l|}
		\hline
		$g$	& $n$	& $\multiplication{m}$	& $G$ & $k$ \\
		\hline
		\hline
		$3$	& $2$	& $\multiplication{2}$	& $\isogenytype{2}{3}$ & $\QQ(\sqrt{-7})$ \\
		\hline
		$5$	& $1$	& $\multiplication{3}$	& $\isogenytype{3}{5}$ & $\QQ(\sqrt{-11})$ \\
		\hline
		$6$	& $3$	& $\multiplication{2}$	& $\isogenytype{2}{6}$ & $\QQ(\sqrt{-7})$ \\
		\hline
		$6$	& $2$	& $\multiplication{3}$	& $\isogenytype{3}{6}$ & $\QQ(\sqrt{-3\sqrt{13}+1})$ \\
		\hline
		$7$	& $2$	& $\multiplication{4}$	& $\isogenytypetwo{4}{4}{2}{6}$ & $\QQ(\sqrt{-15})$ \\
		\hline
		$10$	& $2$	& $\multiplication{3}$	& $\isogenytype{3}{10}$ & $\QQ(\sqrt{-11})$ \\
		\hline
		$10$	& $1$	& $\multiplication{4}$	& $\isogenytypetwo{4}{9}{2}{2}$ & $\QQ(\sqrt{-7})$ \\
		\hline
		$12$	& $3$	& $\multiplication{3}$	& $\isogenytype{3}{12}$ & $\QQ(\sqrt{-3\sqrt{13}+1})$ \\
		\hline
		$14$	& $3$	& $\multiplication{4}$	& $\isogenytypetwo{4}{9}{2}{10}$ & $\QQ(\sqrt{-15})$ \\
		\hline
		$15$	& $1$	& $\multiplication{8}$	& $\isogenytypethree{8}{5}{4}{10}{2}{10}$ & Sextic CM-field (see Ex.~\ref{example:degree-31})  \\
		\hline
		$20$	& $2$	& $\multiplication{4}$	& $\isogenytypetwo{4}{19}{2}{2}$ & $\QQ(\sqrt{-7})$	 \\
		\hline
		$30$	& $2$	& $\multiplication{8}$	& $\isogenytypethree{8}{11}{4}{19}{2}{19}$ & Sextic CM-field (see Ex.~\ref{example:degree-31}) \\
		\hline
	\end{tabular}
	\end{center}
\end{theorem}

Our families of isogenies are derived from
the remarkable explicit classification
of pairs of polynomials $(f,g)$ such that $f(x_1) - g(x_2)$ is reducible
due to 
Cassou--Nog\`ues and Couveignes~\cite{CNC},
building upon the work of 
Fried~\cite{Fried-1,Fried-2,Fried-3}, 
Feit~\cite{Feit-1,Feit-2,Feit-3},
and others~\cite{Cassels}.
We associate a family of pairs of curves 
to every such pair $(f,g)$,
and a family of explicit homomorphisms 
(between the Jacobians of the curves of each pair)
to each factor of $f(x_1) - g(x_2)$.
We show that each homomorphism is in fact an isogeny
of (generically) absolutely simple Jacobians,
and compute the isomorphism type of the kernels.
We also calculate the dimension of the image of each family
in its appropriate moduli space.

Over the complex field,
abelian varieties 
are complex tori,
and we may construct isogenies by
working with period matrices. 
Over arbitrary fields,
these methods are not available to us;
the only abelian varieties for which we have a convenient representation
for explicit computation
are Jacobians of curves,
where we can use the standard isomorphism with the divisor class group.
However,
the Jacobians occupy a positive-codimension subspace of the moduli space
of abelian varieties
in dimension greater than three,
so an isogeny with a Jacobian for a domain generally does not have another
Jacobian for a codomain.
For this reason,
examples of explicit isogenies of higher-dimensional abelian varieties
are particularly rare
(setting aside endomorphisms such as integer multiplication and Frobenius).
We note that recently, 
Mestre has described a $(g+1)$-dimensional family
of $\isogenytype{2}{g}$-isogenies of Jacobians 
of hyperelliptic curves of genus~$g$
for every $g \ge 1$ (see~\cite{Mestre}).
Our families of isogenies 
are defined over number fields,
and provide a source of examples
of explicit isogenies of high-dimensional abelian varieties
over exact fields.

This work generalizes some results from 
the author's unpublished thesis~\cite[\S6]{Smith-thesis}.
The subfamily at $s =0$ of the isogeny in Example~\ref{example:linear-7}
and the fibre at $s = 0$ of the isogeny in Example~\ref{example:linear-11}
also appeared earlier in the thesis of Kux~\cite[\S4.1]{Kux-thesis}.
The Real Multiplication families in Examples~\ref{example:reducible-cyclotomic}
and~\ref{example:TTV-curves}
appeared in the work of 
Tautz, Top, and Verberkmoes~\cite{Tautz-Top-Verberkmoes}.
We assume some familiarity
with the basic theory of curves and abelian varieties,
referring the reader to
Birkenhake and Lange~\cite{Birkenhake-Lange}, 
Hindry and Silverman~\cite[Part~A]{Hindry-Silverman},
Milne~\cite{Milne},
and Shimura~\cite{Shimura}
for further details.

\subsection*{Notation}
Throughout this article,
$\zeta_n$ denotes 
a primitive $n^\mathrm{th}$ root of unity in $\closure{\QQ}$.
If $\sigma$ is an automorphism of a field $k$
and $f(x) = \sum_{i} c_i x^i$ is a polynomial over $k$,
then we write $f^\sigma(x)$
for the polynomial $\sum_{i} c_i^\sigma x^i$.
If $\phi$ is an isogeny of abelian varieties
with kernel isomorphic to a group~$G$,
then we say that $\phi$ is a $G$-isogeny.

\subsection*{Acknowledgements}
The author is grateful to the mathematics departments
of the University of Sydney and Royal Holloway, University of London,
where parts of this work were carried out.
This research was supported in part by EPSRC grant EP/C014839/1.

\section{%%%%%%%%%%%%%%%%%%%%%%%%%%%%%%%%%%%%%%%%%%%%%%%%%%%%%%%%%%%%%%%%%%%%%%%
	The basic construction
}%%%%%%%%%%%%%%%%%%%%%%%%%%%%%%%%%%%%%%%%%%%%%%%%%%%%%%%%%%%%%%%%%%%%%%%%%%%%%%%
\label{section:basic-construction}

Suppose $(f,g)$ is a pair of squarefree polynomials 
of degree at least~$5$ over $k$
such that there exists a nontrivial factorization
\[
	f(x_1) - g(x_2) = A(x_1,x_2)B(x_1,x_2) .
\]
Given such a pair of polynomials,
we define a pair $(X,Y)$ of hyperelliptic curves by
\[
	X : y_1^2 = f(x_1) \quad \text{and}\quad Y : y_2^2 = g(x_2) 
	.
\]
The factors $A$ and $B$ of $f(x_1) - g(x_2)$
define explicit homomorphisms from $\Jac{X}$ to~$\Jac{Y}$
as follows:
Let $C$ be the correspondence on $\XxY$
defined by
\begin{equation}
\label{eq:correspondence}
	C := \variety{y_1 - y_2, A(x_1,x_2)} \subset \XxY .
\end{equation}
The natural projections of $\XxY$
restrict to coverings $\pi_X: C \to X$
and $\pi_Y: C \to Y$.
Composing the pullback $(\pi_X)^*: \Jac{X} \to \Jac{C}$
with the pushforward $(\pi_Y)_*: \Jac{C} \to \Jac{Y}$,
we obtain a homomorphism
\begin{equation}
\label{eq:homomorphism}
	\phi := \compose{(\pi_Y)_*}{(\pi_X)^*} 
	: 
	\Jac{X} \longrightarrow \Jac{Y} 
	;
\end{equation}
we say $\phi$ is \emph{induced} by $C$.
The homomorphism $\phi$ is completely explicit:
we can compute the image of a divisor class on $X$ under $\phi$
by pulling back a representative divisor to $C$ 
and then pushing the result forward onto $Y$.
%The methods of~\cite{Kohel-Smith}
%can easily be adapted 
%to express $\phi$
%as an efficiently computable map on Mumford ideal class representatives.

If we replace $A$ with $B$ in~\eqref{eq:correspondence},
%then the resulting correspondence induces the homomorphism $-\phi$.
we obtain the homomorphism $-\phi$.
Exchanging $X$ and $Y$ in~\eqref{eq:homomorphism},
we obtain the Rosati dual homomorphism $\dualof{\phi}: \Jac{Y} \to \Jac{X}$
(recall $\dualof{\phi} = \lambda_X^{-1}\hat{\phi}\lambda_Y$,
where $\hat\phi: \Jacdual{Y} \to \Jacdual{X}$ 
is the dual homomorphism
and $\lambda_X: \Jac{X} \stackrel{\sim}{\to} \Jacdual{X}$ 
and $\lambda_Y: \Jac{Y} \stackrel{\sim}{\to} \Jacdual{Y}$
are the canonical principal polarizations).

If $A(x_1,x_2)$ divides $f(x_1) - g(x_2)$,
then
it also divides $F(f(x_1)) - F(g(x_2))$
for every polynomial $F$ over $k$.
Therefore, 
if we let 
$F = x^d + s_1x^{d-1} + \cdots + s_{d-1}x + s_d$
be the generic monic polynomial of degree $d$
(where the $s_i$ are free parameters),
let $\Delta_f$ (resp.~$\Delta_g$) be the discriminant of $F(f(x))$ (resp.~$F(g(x))$),
and let $T$ be the parameter space defined by
\[
	T := 
	\Spec(k[s_1,\ldots,s_d])
	\setminus
	\left(\variety{\Delta_f}\cup\variety{\Delta_g}\right)
	,
\]
then we obtain a $d$-parameter family $(\family{X},\family{Y}) \to T$
of pairs of curves defined by
\[
	\family{X}: 
	y_1^2 
	= F(f(x_1)) 
	= f(x_1)^d + s_1f(x_1)^{d-1} + \cdots + s_{d-1}f(x_1) + s_d 
\]
and
\[
	\family{Y}: 
	y_2^2 
	= 
	F(g(x_2)) 
	= g(x_2)^d + s_1g(x_2)^{d-1} + \cdots + s_{d-1}g(x_2) + s_d 
	,
\]
together with a family of homomorphisms 
$\phi: \Jacfamily{\family{X}} \to \Jacfamily{\family{Y}}$
induced by the correspondence
\[
	C 
	= 
	\variety{y_1 - y_2, A(x_1,x_2)} \subset \product{\family{X}}{\family{Y}}
	.
\]
That is, 
for each $P$ in $T$,
if $C_P$, $X_P$, and $Y_P$
are the fibres of 
$C$, $\family{X}$, and $\family{Y}$ over $P$,
then $C_P$ induces a homomorphism $\phi_P: \Jac{X_P} \to \Jac{Y_P}$.

If $f$ and $g$ are defined over a polynomial ring $k[t]$,
then we can define $(d+1)$-parameter families of pairs of curves and homomorphisms,
this time parameterised by 
$T = \Spec(k[t,s_1,\ldots,s_d])\setminus \left(\variety{\Delta_f}\cup\variety{\Delta_g}\right)$,
in exactly the same way.
Throughout this article we will use $T$
to denote the parameter space of each of our families;
the precise definition of $T$ in each case will be clear
from the context.

We will restrict our attention to the cases $d = \deg F = 1$ 
(the \emph{linear construction})
and $d = \deg F = 2$
(the \emph{quadratic construction}).
For higher degrees $d$,
the Jacobians of $\family{X}$ and $\family{Y}$
are reducible.
Indeed, we have a covering $(x,y)\mapsto(f(x),y)$
from $\family{X}$ to the curve $\family{X}': v^2 = F(u)$,
so $\Jacfamily{\family{X}'}$
is an isogeny factor of $\Jacfamily{\family{X}}$
whenever $\family{X}'$ has positive genus:
that is, whenever $d > 2$.
We aim to construct
explicit isogenies of absolutely simple Jacobians,
so
we will leave aside $d > 2$.

\begin{remark}
	Our constructions depend only on $f$ and $g$,
	and generalize to variable-separated curves
	--- 
	that is, curves of the form $P(y) = f(x)$ 
	where $\deg P > 2$.
	The analysis of the resulting homomorphisms
	is more detailed, however,
	and some of the methods we use in \S\ref{section:kernel-structure}
	do not readily extend to the separated-variable case.
	We will return to these constructions in future work.
\end{remark}

\section{%%%%%%%%%%%%%%%%%%%%%%%%%%%%%%%%%%%%%%%%%%%%%%%%%%%%%%%%%%%%%%%%%%%%%%%
	Determining kernel structure
}%%%%%%%%%%%%%%%%%%%%%%%%%%%%%%%%%%%%%%%%%%%%%%%%%%%%%%%%%%%%%%%%%%%%%%%%%%%%%%%
\label{section:kernel-structure}

% FIXME: Masser's "Specialization of endomorphism rings of abelian varieties" might be useful in supporting the specialization arguments...

Suppose $\family{X}$, $\family{Y}$, and $\phi: \Jacfamily{X} \to \Jacfamily{Y}$
are defined as in the previous section.
We want to determine whether $\phi$ is an isogeny,
and if so to compute a group $G$ isomorphic to its kernel.
It suffices to consider the generic fibre $\phi: \Jac{X} \to \Jac{Y}$,
which is defined over $\QQbar(T)$, 
a field of characteristic zero.
%We will use two techniques:
%examining the action of $\phi$ on differentials,
%and computing the induced homomorphism of two-torsion subgroups.

The first step is to show that $\Jac{X}$ is absolutely simple;
then $\phi$ is an isogeny if and only if it is nonzero.
Further, if $\phi$ is an isogeny and $\Jac{X}$ is absolutely simple
and $\genus{X} = \genus{Y}$
then $\Jac{Y}$ must also be absolutely simple,
and $\phi$ itself cannot arise from a product of isogenies
of lower-dimensional abelian varieties.
Since a reducible abelian variety cannot specialize to to
an absolutely simple one,
it is enough to exhibit a point $P$ of the parameter space~$T$
such that the specialization $J_P$ of $\Jacfamily{X}$ at $P$
is absolutely simple.
In 
Examples~\ref{example:linear-7} and~\ref{example:linear-13},
there will exist a convenient choice of $P$
allowing us to deduce the simplicity of $J_P$ from CM-theory.
For the other examples,
we will use the fact that $J_P$ is defined over a number field~$K$,
and exhibit a prime $\mathfrak{p}$ of~$K$
such that the (good) reduction $\overline{J_P}$ of~$J_P$ at~$\mathfrak{p}$
is absolutely simple;
the absolute simplicity of $J_P$,
and thus the absolute simplicity of $\Jac{X}$,
then follows from~\cite[Lemma 6]{Chai--Oort}. 

To show that $\overline{J_P}$ is absolutely simple,
we compute its Weil polynomial $\chi$
(that is, the characteristic polynomial of its Frobenius endomorphism)
using Kedlaya's algorithm~\cite{Kedlaya}, 
which is implemented in Magma~\cite{Harrison, Magma}.
(For this to be practical the norm of $\mathfrak{p}$ must be a power of a small prime,
especially for the higher-genus families.)
If $\chi$ is irreducible,
then $\overline{J_P}$ is simple.
To determine whether $\overline{J_P}$ is \emph{absolutely} simple,
we apply
the criterion appearing in~\cite{Howe-Zhu}:

\begin{lemma}[Howe and Zhu]
\label{lemma:simplicity-criterion}
	Suppose $A$
	is a simple abelian variety over a finite field,
	and let $\chi$
	be its (irreducible) Weil polynomial.
	Let $\pi$ be an element of $\QQbar$
	satisfying $\chi(\pi) = 0$.
	Let $D$ be the set of integers $d > 1$
	such that either
	\begin{enumerate}
		\item	$\chi(x)$ lies in $\ZZ[x^d]$, or
		\item	$[\QQ(\pi):\QQ(\pi^d)] > 1$
			and $\QQ(\pi) = \QQ(\pi^d,\zeta_d)$.
	\end{enumerate}
	If $D$ is empty,
	then $A$ is absolutely simple.
\end{lemma}
\begin{proof}
	See~\cite[Proposition 3]{Howe-Zhu}.
	Note that 
	\(D \subset \{ d \in \ZZ_{>0} : \varphi(d) \mid 2\dim A \}\),
	so this criterion can be efficiently checked.
\end{proof}

We will be handling some large Weil polynomials.
To save space, we will use the following, more compact representation.

\begin{definition}
\label{definition:Weil-coefficients}
	Suppose $A$ is a $g$-dimensional abelian variety over $\FF_{q}$
	with Weil polynomial $\chi$.
	We define the \emph{Weil coefficients} of $A$ to be the integers
	$w_1, \ldots, w_g$
	such that
	\[
		\chi(x)
		=
		x^{2g} + w_1 t^{2g-1} + \cdots + w_g x^g
		+ w_{g-1}qx^{g-1} + \cdots + w_1q^{g-1}x + q^{g} .
	\]
\end{definition}

Recall that 
$\compose{\dualof{\phi}}{\phi}$
is an endomorphism of $\Jacfamily{X}$;
if $\phi$ is an isogeny of absolutely simple Jacobians,
then $\compose{\dualof{\phi}}{\phi} = \multiplication[\Jacfamily{X}]{m}$
for some nonzero integer $m$.
Conversely,
if~$\compose{\dualof{\phi}}{\phi} = \multiplication[\Jacfamily{X}]{m}$ for some $m$,
then $\phi$ is an isogeny and $\ker\phi\subset\Jacfamily{X}[m]$.
Since $\phi$ is an isogeny of Jacobians
(thus respecting the canonical polarizations),
its kernel must be 
a maximal isotropic subgroup of $\Jacfamily{X}[m]$
with respect to the $m$-Weil pairing;
the nondegeneracy of the Weil pairing then
gives the following elementary result.

\begin{lemma}
\label{lemma:kernel-structure}
	If $\compose{\dualof{\phi}}{\phi} = \multiplication[\Jacfamily{X}]{m}$
	for some positive integer $m$,
	then 
	$\phi$ is a $G$-isogeny
	for some subgroup $G$ of $(\ZZ/m\ZZ)^{2\genus{X}}$
	such that $G \cong (\ZZ/m\ZZ)^{2\genus{X}}/G$.
	Further, if $m$ is squarefree,
	then $G \cong (\ZZ/m\ZZ)^{\genus{X}}$.
\end{lemma}

Let $K = \QQbar(T)$
denote the base field of the generic fibre,
and
let $\differentials(X)$ and $\differentials(Y)$
denote
the $K$-vector spaces of regular differentials on~$X$ and~$Y$,
respectively.
We have the well-known representation
\[
	\differentialmatrix[X,Y]{\cdot}:
	\Hom(\Jac{X},\Jac{Y})
	\longrightarrow
	\Hom(\differentials(X),\differentials(Y))
	,
\]
sending a homomorphism to the induced map on differentials
(see Shimura~\cite[\S2.9]{Shimura} for details).
This representation is faithful in characteristic zero,
and it respects composition:
if $\phi: \Jac{X} \to \Jac{Y}$ 
and $\psi: \Jac{Y} \to \Jac{Z}$
are homomorphisms,
then
\[
	\differentialmatrix[X,Z]{\psi\circ\phi}
	=
	\differentialmatrix[X,Y]{\phi}
	\differentialmatrix[Y,Z]{\psi}
	.
\]
In particular, when $\Jac{X} \cong \Jac{Y}$
we obtain a representation of rings
\[
	\differentialmatrix[X]{\cdot}
	: 
	\End(\Jac{X}) 
	\longrightarrow 
	\End(\differentials(X))
	.
\]
To
determine whether $\phi$ is an isogeny,
we compute
$
	\differentialmatrix[X]{{\dualof{\phi}}{\phi}} 
	= 
	\differentialmatrix[X,Y]{\phi}
	\differentialmatrix[Y,X]{\dualof{\phi}}
$
and check that the result is equal to $mI_\genus{X}$
for some integer $m \not= 0$.
Given $m$,
we can use Lemma~\ref{lemma:kernel-structure}
to partially determine the group structure of $\ker\phi$.

It is straightforward to compute 
$\differentialmatrix[X,Y]{\phi}$
when $\phi$ is induced by a correspondence 
of the form $C = \variety{y_1 -  y_2, A(x_1,x_2)} \subset \XxY$ .
We begin by fixing ordered bases
\[
	\differentials(X)
	=
	\left\langle 
		d(x_1^i)/y_1 : 1 \le i \le \genus{X} 
	\right\rangle
	\text{\quad and\quad }
	\differentials(Y)
	=
	\left\langle
		d(x_2^i)/y_2 : 1 \le i \le \genus{Y}
	\right\rangle
\]
for $\differentials(X)$ and $\differentials(Y)$.
Then $\differentialmatrix[X,Y]{\cdot}$ 
becomes a representation
into $\Mat_{\genus{X}\times\genus{Y}}(K)$
%$
	%\differentialmatrix[X,Y]{\cdot}: 
	%\Hom(\Jac{X},\Jac{Y}) 
	%\to 
	%\Mat_{\genus{X}\times\genus{Y}}(k) 
%$
(viewing elements of $\differentials(X)$ and $\differentials(Y)$
as row vectors,
with
matrices representing homomorphisms 
act by multiplication on the right.)
%We call $\differentialmatrix[X,Y]{\phi}$ 
%the \emph{differential matrix} of $\phi$.
Pulling back our basis 
of~$\differentials(X)$ to~$\differentials(C)$
(via the inclusion $K(X)\hookrightarrow K(C)$ induced by $\pi_X$)
and then taking the trace from $\differentials(C)$ to $\differentials(Y)$
(with respect to the inclusion $K(Y)\hookrightarrow K(C)$ induced by $\pi_Y$),
we have
\[
	\phi_*(d(x_1^i)/y_1)
	=
	\Tr^{\differentials(C)}_{\differentials(Y)}(d(x_1^i)/y_1)
	=
	dt_i/y_2 ,
\]
where $t_i$ is the trace from $K(C)$ to $K(Y)$ of $x_1^i$.
To compute these traces,
we rewrite $A(x_1,x_2)$ as a polynomial in $x_1$ over $K[x_2]$
(after possibly rescaling to ensure $A$ is monic in $x_1$):
\[
	A(x_1,x_2)
	=
	x_1^d + \sum_{i=1}^{d} (-1)^{i}s_i(x_2)x_1^{d-i}
	.
\]
The $s_i$ are the $i^\mathrm{th}$
elementary symmetric polynomials in the roots of $A$
viewed as a polynomial in $x_1$ over $K(x_2)$;
each $s_i$ is a polynomial in $x_2$ over $K$
of degree at most $i$.
The function $t_i$ is by definition
the $i^\mathrm{th}$ power sum symmetric function
in these same roots,
and so we can express the $t_i$
in terms of the $s_j$
using the standard Newton-Girard recurrences:
\[
	k s_k = \sum_{i=1}^k (-1)^{i-1} s_{k-i} t_{i} .
\]
Since each $s_i$ has degree at most $i$,
it follows that each of the trace functions $t_i$
has degree at most $i$.
We can therefore write
\[
	d(t_i)/y_2 = \sum_{j=1}^{i} t_{i,j} d(x_2^j)/y_2
\]
with coefficients $t_{i,j}$ in $K$;
these coefficients 
are precisely the entries of~$\differentialmatrix[X,Y]{\phi}$
(with $t_{i,j} = 0$ for $j > i$).

We noted above that if $\phi: \Jac{X}\to\Jac{Y}$
is a homomorphism induced by a correspondence on $\XxY$,
then we obtain the Rosati dual $\dualof{\phi}: \Jac{Y}\to\Jac{X}$
by simply exchanging~$X$ and $Y$. % in~\eqref{eq:homomorphism}.
We may therefore compute $\differentialmatrix[Y,X]{\dualof{\phi}}$
in exactly the same way we computed $\differentialmatrix[X,Y]{\phi}$,
expressing the differentials 
$\dualof{\phi}_*(d(x_2^i)/y_2) = \Tr^{\differentials(C)}_{\differentials(X)}(d(x_2^i)/y_2)$
as linear combinations of the $d(x_1^j)/y_1$.

If $\compose{\dualof{\phi}}{\phi} = \multiplication[\Jac{X}]{m}$,
then Lemma~\ref{lemma:kernel-structure}
allows us to determine the structure of $\ker\phi$
when $m$ is squarefree.
But in \S\ref{sec:examples}
we will encounter $m = 2$, $3$, $4$, and $8$;
we will therefore need another technique to handle
$m = 4$ and $m = 8$.

\begin{lemma}
\label{lemma:kernel-from-two-rank}
	Let 
	$\phi:\Jac{X} \to \Jac{Y}$
	be an isogeny
	over a field of characteristic not $2$,
	such that $\dualof{\phi}\phi = \multiplication[X]{m}$
	with $m = 4$ or~$8$,
	and 
	let $\nu$ be the $(\ZZ/2\ZZ)$-rank
	of $\ker\phi\cap\Jac{X}[2]$.
	\begin{enumerate}
		\item	If $m = 4$,
			then $\ker\phi \cong \isogenytypetwo{4}{2\genus{X}-\nu}{2}{2(\nu - \genus{X})}$.
		\item	If $m = 8$,
			then $\ker\phi \cong \isogenytypethree{8}{2\genus{X}-\nu}{4}{\nu-\genus{X}}{2}{\nu - \genus{X}}$.
	\end{enumerate}
\end{lemma}
\begin{proof}
	The result follows directly from Lemma~\ref{lemma:kernel-structure}.
\end{proof}

To apply Lemma~\ref{lemma:kernel-from-two-rank},
we need to compute the $(\ZZ/2\ZZ)$-rank $\nu$ of $\ker\phi\cap\Jac{X}[2]$.

\begin{lemma}
\label{lemma:two-rank}
	Let $f(x) = \prod_{i=1}^d(x - \gamma_i)$ 
	and $g(x) = \prod_{i=1}^d(x - \delta_i)$
	be polynomials of degree $d > 2$
	over a field of characteristic not $2$
	such that $f(x_1) - g(x_2)$
	has a nontrivial factor $A(x_1,x_2)$.
	Let
	$X: y_1^2 = f(x_1)$
	and $Y: y_2^2 = g(x_2)$
	be hyperelliptic curves,
	and
	$\phi: \Jac{X} \to \Jac{Y}$
	the homomorphism induced by
	the correspondence
	$\variety{y_1-y_2,A(x_1,x_2)}$ on $\XxY$.
	The $(\ZZ/2\ZZ)$-rank of $\ker\phi\cap\Jac{X}[2]$
	is given by
	\[
		\mathrm{rank}_{(\ZZ/2\ZZ)}( \ker\phi \cap \Jac{X}[2] ) 
		=
		\dim( \ker M ),
	\]
	where $M$ is the $2\genus{X}\!\times\!2\genus{Y}$ matrix over $\FF_{2}$
	with $i,j$-th entry $\nu_{i,j} + \nu_{i,2\genus{Y}+1} \pmod{2}$,
	where $\nu_{i,j}$ denotes the multiplicity of $(x_2 - \delta_j)$
	as factor of $A(\gamma_i,x_2)$
	for $1 \le i , j \le d$.
\end{lemma}
\begin{proof}
	For each $1 \le i \le d$,
	we let $w_i$ be the point $(\gamma_i,0)$ on $X$
	and let $w_i'$ be the point $(\delta_i,0)$ on $Y$.
	If $d$ is odd (so $d = 2\genus{X} + 1 = 2\genus{Y} + 1$), 
	then we let $w_{2\genus{X} + 2}$ (resp. $w_{2\genus{Y} + 2}'$)
	be the unique point at infinity on $X$ (resp. $Y$),
	and we set $\nu_{i,2\genus{Y} + 2} := 0$
	for $1 \le i \le 2\genus{X}$.
	The sets $\{ w_i : 1 \le i \le 2\genus{X} + 2\}$
	and $\{ w_i' : 1 \le i \le 2\genus{X} + 2\}$
	are then the sets of Weierstrass points
	of $X$ and $Y$, respectively.
	It is well-known that $\Jac{X}[2]$ (resp.~$\Jac{Y}[2]$)
	is generated by differences of Weierstrass points of $X$
	(resp.~$Y$),
	subject to the relations
	\[
		\begin{array}{r@{\;=\;}l}
		{}[ (w_{2\genus{X} + 1}) - (w_{2\genus{X}+2}) ]
		& 
		\sum_{i=1}^{2\genus{X}}[ (w_i) - (w_{2\genus{X}+2}) ] 
		\quad \text{and}
		\\
		{}[ (w_{2\genus{Y} + 1}') - (w_{2\genus{Y}+2}') ] 
		&
		\sum_{i=1}^{2\genus{Y}}[ (w_i') - (w_{2\genus{Y}+2}') ] 
		.
		\end{array}
	\]
	We therefore fix explicit bases for the $2$-torsion:
	\[
		\Jac{X}[2] 
		=
		\left\langle [ (w_i) - (w_{2\genus{X} + 2}) ]  \right\rangle_{i = 1}^{2\genus{X}}
		\text{\quad and\quad }
		\Jac{Y}[2]
		=
		\left\langle [ (w_i') - (w_{2\genus{Y} + 2}') ] \right\rangle_{i = 1}^{2\genus{Y}}
		.
	\]
	Since ${\phi}$
	restricts to a homomorphism ${\phi}|_2: \Jac{X}[2] \to \Jac{Y}[2]$,
	we have a representation
	\[
		T_2(\cdot) 
		: 
		\Hom(\Jac{X},\Jac{Y}) 
		\longrightarrow 
		\Hom(\Jac{X}[2],\Jac{Y}[2]) 
		\cong
		\Mat_{2\genus{X}\times2\genus{Y}}(\FF_{2})
	\]
	(where the isomorphism is determined by our choice of bases.)
	The $(\ZZ/2\ZZ)$-rank of 
	$\ker{\phi}\cap\Jac{X}[2] = \ker{\phi}|_2$
	is then equal to
	the nullity of the matrix $T_2({\phi})$.
	The entries $t_{i,j}$ 
	of $\torsionmatrix{\phi}$
	are determined by the relations
	\[
		\phi([(w_i) - (w_{2\genus{X} + 2})])
		= 
		\sum_{j=1}^{2\genus{Y}} t_{i,j}[(w_j') - (w_{2\genus{Y} + 2}')]
	\]
	(this is well-defined, since the $t_{i,j}$ are elements of $\ZZ/2\ZZ$).
	Explicitly computing the images of the basis elements, we find
	\[
		\begin{array}{r@{\;}c@{\;}l}
			\phi([(w_i) - (w_{2\genus{X} + 2})])
			& = &
			\sum_{j=1}^{2\genus{Y}+1}(\nu_{i,j} - \nu_{i,{2\genus{Y} + 2}})[(w_i')-(w_{2\genus{Y} + 2}')]
			\\ & = &
			\left( \sum_{j=1}^{2\genus{Y}}(\nu_{i,j} - \nu_{i,{2\genus{Y} + 2}})[(w_i')-(w_{2\genus{Y} + 2}')] \right)
			\\ &   & {}
			+ (\nu_{i,2\genus{Y}+1} - \nu_{i,2\genus{Y} + 2})\left(\sum_{j=1}^{2\genus{Y}}[ (w_i') - (w_{2g+2}') ] \right)
			\\ & = &
			\sum_{j=1}^{2\genus{Y}}(\nu_{i,j} + \nu_{i,2\genus{Y}+1} - 2\nu_{i,{2\genus{Y} + 2}})[(w_i')-(w_{2\genus{Y} + 2}')]
			\\ & = &
			\sum_{j=1}^{2\genus{Y}}(\nu_{i,j} + \nu_{i,2\genus{Y}+1})[(w_i')-(w_{2\genus{Y} + 2}')] ,
		\end{array}
	\]
	so
	\( t_{i,j} \equiv \nu_{i,j} + \nu_{i,2g+1} \pmod 2 \)
	for
	\( 1 \le j \le 2\genus{Y} \)
	and \(1 \le i \le 2\genus{X} \).
	Hence $M = T_2(\phi)$, and the result follows.
\end{proof}

\begin{remark} 
	In practice,
	computing the matrix $M$ of Lemma~\ref{lemma:two-rank}
	can be difficult
	if the roots of $f$ and $g$ 
	are not all defined over a low-degree extension of the ground field.
	In our examples, we will be free to choose (reductions of) $X$ and $Y$ 
	in such a way that all of the roots of $f$ and $g$ lie in
	a small finite field.
\end{remark}

\section{%%%%%%%%%%%%%%%%%%%%%%%%%%%%%%%%%%%%%%%%%%%%%%%%%%%%%%%%%%%%%%%%%%%%%%
	Pairs of polynomials 
}%%%%%%%%%%%%%%%%%%
\label{section:polynomial-classification}

In order to use the construction of \S\ref{section:basic-construction}
to produce examples of explicit isogenies,
we need a source of
pairs of polynomials $(f,g)$
such that $f(x_1) - g(x_2)$ is reducible.
We will use the explicit classification of such pairs over $\CC$
due to Cassou--Nogu\`es and Couveignes~\cite{CNC},
which we summarize in Theorem~\ref{th:CNC-theorem}.
This classification is restricted to indecomposable polynomials
(in the sense of Definition~\ref{def:indecomposable}),
and classifies pairs up to an equivalence relation
described in Definition~\ref{def:equivalence}.

%and classifies pairs up to a certain equivalence relation,
%which we describe in Definition~\ref{def:equivalence}.
\begin{definition}
\label{def:equivalence}
	We say that polynomials $f_1$ and $f_2$ over $k$ are
	\emph{linear translates}
	if there exist $a$ and $b$ in $\closure{k}$,
	with $a \not= 0$,
	such that
	$f_1(x) = f_2(ax + b)$.
	We say pairs $(f_1,g_1)$ and $(f_2,g_2)$ of polynomials are
	\emph{equivalent}
	%\footnote{
		%\emph{faiblement lin\'earement reli\'es} 
		%(weakly linearly connected) in~\cite{CNC}.
	%}
	if there exists $c \not= 0$ and $d$ in $\closure{k}$ 
	such that $f_1$ and $cf_2 + d$ are linear translates
	and $g_1$ and $cg_2 + d$ are linear translates.
\end{definition}

%One easily checks that the equivalence of Definition~\ref{def:equivalence}
%is indeed an equivalence relation
The ``equivalence'' of Definition~\ref{def:equivalence}
is indeed an equivalence relation
on pairs of polynomials.
Further,
if $\mathcal{S}$ is an equivalence class,
then either $f(x_1) - g(x_2)$ is $\closure{k}$-reducible for each $(f,g)$ in $\mathcal{S}$
or $f(x_1) - g(x_2)$ is $\closure{k}$-irreducible for each $(f,g)$ in~$\mathcal{S}$.

\begin{definition}
\label{def:indecomposable}
        A polynomial $f$ is \emph{decomposable}
        if 
        $f(x) = f_1(f_2(x))$
	for some polynomials $f_1$ and $f_2$
        of degree at least $2$,
        and \emph{indecomposable} otherwise.
\end{definition}

%	FIXME:
%	We have a complete classification of indecomposable pairs.
%	This is convenient.
%	There may be more pairs of the form $(F(f(x_1)),G(g(x_1)))$,
%	but the theory of factorization for these pairs is not clear...

\begin{theorem}[Cassou--Nogu\`es and Couveignes~\cite{CNC}]
\label{th:CNC-theorem}
	Let $(f,g)$ be a pair of indecomposable polynomials 
	of degree at least~$3$ over $\CC$,
	and let $\sigma$ denote complex conjugation.
	Assume the classification of finite simple groups.

	If $f$ and $g$ are linear translates,
	then $f(x_1) - g(x_2)$ is divisible by $x_1 - x_2$,
	and $(f(x_1) - g(x_2))/(x_1-x_2)$
	is reducible if and only if 
	$(f,g)$ is equivalent to either
	\begin{enumerate}
		\item	the pair $(x^n,x^n)$
			for some prime $n$, or
		\item	the pair $(D_n(x),D_n(x))$
			for some prime $n$,
			where $D_n(x)$ is 
			defined in Example~\ref{example:Dickson-polynomials}.
	\end{enumerate}
	If $f$ and $g$ are not linear translates,
	then $f(x_1) - g(x_2)$ is reducible
	if and only if 
	$(f,g)$ is equivalent to one of the following
	(possibly after exchanging $f$ and $g$):
	\begin{enumerate}
		\setcounter{enumi}{2}
		\item	a pair in the 
			one-parameter family $(f_{7},f_{7}^\sigma)$
			defined in Example~\ref{example:degree-7},
			or
		\item	the pair $(f_{11},f_{11}^\sigma)$
			defined in Example~\ref{example:degree-11},
			or
		\item	a pair in the
			one-parameter family $(f_{13},f_{13}^\sigma)$
			defined in Example~\ref{example:degree-13},
			or
		\item	a pair in the
			one-parameter family $(f_{15},-f_{15}^\sigma)$
			%(note the negative)
			defined in Example~\ref{example:degree-15},
			or
		\item	the pair $(f_{21},f_{21}^\sigma)$
			defined in Example~\ref{example:degree-21},
			or
		\item	the pair $(f_{31},f_{31}^\sigma)$
			defined in Example~\ref{example:degree-31}.
	\end{enumerate}
\end{theorem}

\begin{example}[Cyclic polynomials]
\label{example:cyclic-polynomials}
	The difference $x_1^n - x_2^n$
	factors as
	\[
		x_1^n - x_2^n = \prod_{e=0}^{n-1}(x_1 - \zeta_n^ex_2) 
		.
	\]
\end{example}

\begin{example}[Dickson polynomials]
\label{example:Dickson-polynomials}
	For each $n \ge 1$,
	we let $D_n(x) = D_n(x,1)$ denote 
	the $n^\mathrm{th}$ Dickson polynomial 
	of the first kind with parameter $1$:
	that is, 
	the unique polynomial of degree $n$
	such that $D_n(x + x^{-1},1) = x^n + x^{-n}$.
	In characteristic zero
	we have $D_n(x) = 2T_n(x/2)$,
	where $T_n$ is the classical Chebyshev polynomial
	of degree $n$.
	(See~\cite{Lidl-Mullen-Turnwald} for 
	further details.)
	We have a nontrivial factorization
	\[
		D_n(x_1) - D_n(x_2)
		=
		(x_1 - x_2)
		\!\!\!\!\prod_{i=1}^{(n-1)/2}\!\!\!\!
		A_{n,i}(x_1,x_2)
	\]
	(see \cite[Theorem~3.12]{Lidl-Mullen-Turnwald}),
	where
	\[
		A_{n,i}(x_1,x_2) 
		:= 
		x_1^2 + x_2^2 
		- (\zeta_n^i + \zeta_n^{-i})x_1x_2 
		+ (\zeta_n^i - \zeta_n^{-i})
		.
	\]
\end{example}

\begin{example}[Polynomials of degree $7$]
%*{Families \Lseven{}  and \Qseven}%%%%%%%%%%%%%%%%%%%%%%
%\label{subsection:genus-three}
\label{example:degree-7}
	Let 
	$\alpha_7$ be an element of $\QQbar$
	satisfying
	\[
		\alpha_7^2 + \alpha_7 + 2 = 0 ,
	\]
	The involution $\sigma: \alpha_7 \mapsto 2/\alpha_7$ 
	generates $\Gal{\QQ(\alpha_7)/\QQ}$.
	Note that $\QQ(\alpha_7) = \QQ(\sqrt{-7})$ 
	is a quadratic imaginary field,
	and $\sigma$ acts as complex conjugation.
	
	Let $f_{7}$ be the polynomial of degree $7$
	over $\QQ(\alpha_7)[t]$
	defined by
	\[
		\begin{array}{r@{\;}l}
			f_{7}(x) 
			:= 
			&
			\frac{1}{7}x^7 
			- \alpha_7t x^5 
			- \alpha_7t x^4 
			- (2 \alpha_7 + 5)t^2 x^3 
			- (4 \alpha_7 + 6)t^2 x^2 
			\\
			&
			{} + ((3 \alpha_7 - 2)t^3 - (\alpha_7 + 3)t^2) x 
			+ \alpha_7t^3
			.
		\end{array}
	\]
	(Our $f_{7}$ is 
	the polynomial of~\cite[\S5.1]{CNC}
	with $a_2 = \alpha_7$.)
	We have a nontrivial factorization 
	$f_{7}(x_1) - f_{7}^\sigma(x_2) = A_7(x_1,x_2)B_7(x_1,x_2)$,
	where
	$$
		A_{7}%(x_1,x_2)
		= 
		x_1^3 - x_2^3 - \alpha_7^\sigma x_1^2x_2 + \alpha_7 x_1x_2^2 
		+ (3-2\alpha_7^\sigma)tx_1 - (3-2\alpha_7)tx_2 + (\alpha_7 - \alpha_7^\sigma)t
		;
	$$
	note that $A_{7}(x_2,x_1) = -A_{7}(x_1,x_2)^\sigma$.
	Both $A_{7}$ and $B_{7}$ 
	are absolutely irreducible.
	%(their coefficients are given in the file \texttt{degree-7.m}).
\end{example}

\begin{example}[Polynomials of degree 11]
\label{example:degree-11}
%\label{subsection:genus-five}
% normalization: 2/alpha
	Let $\alpha_{11}$
	be an element of $\QQbar$
	satisfying
	\[
		\alpha_{11}^2 + \alpha_{11} + 3 = 0 ;
	\]
	the involution $\sigma: \alpha_{11} \mapsto 3/\alpha_{11}$ 
	generates $\Gal{\QQ(\alpha_{11})/\QQ}$.
	Note that $\QQ(\alpha_{11}) = \QQ(\sqrt{-11})$
	is an imaginary quadratic field,
	and $\sigma$ acts as complex conjugation.
	
	Let
	$f_{11}$ be the polynomial of degree $11$ over $\QQ(\alpha_{11})$
	defined by
	\[
		\begin{array}{r@{\;}l}
		f_{11}(x) := 
		&
		\frac{1}{11}x^{11} + \alpha_{11} x^9 + 2 x^8 - 3(\alpha_{11} + 4) x^7 + 16 \alpha_{11} x^6 
		\\
		&
		{} - 3(7\alpha_{11} - 5) x^5 - 30(\alpha_{11} + 4) x^4 + 63 (\alpha_{11} + 1) x^3 
		\\
		&
		{} - 20(5\alpha_{11} - 1) x^2 - 3(8\alpha_{11} + 47) x + 18 \alpha_{11}
		.
		\end{array}
	\]
	(Our $f_{11}$ is 
	the polynomial of~\cite[\S5.2]{CNC}
	with $a_2 = \alpha_{11}^\sigma$.)
	We have a nontrivial factorization
	$f_{11}(x_1) - f_{11}^\sigma(x_2) = A_{11}(x_1,x_2)B_{11}(x_1,x_2)$,
	where 
	\[
		\begin{array}{r@{\;}l}
		A_{11}(x_1,x_2) 
		=
		&
		x_1^5 - \alpha_{11}x_1^4x_2 - x_1^3x_2^2 + (4 \alpha_{11} + 2)x_1^3 + x_1^2x_2^3 + (\alpha_{11} + 6)x_1^2x_2 
		\\ & {}
		- (2 \alpha_{11} - 10)x_1^2 - (\alpha_{11} + 1)x_1x_2^4 + (\alpha_{11} - 5)x_1x_2^2 
		\\ & {}
		- (12 \alpha_{11} + 6)x_1x_2 + (8 \alpha_{11} - 7)x_1 - x_2^5 + (4 \alpha_{11} + 2)x_2^3 
		\\ & {}
		- (2 \alpha_{11} + 12)x_2^2 + (8 \alpha_{11} + 15)x_2 + 12 \alpha_{11} + 6
		;
		\end{array}
	\]
	note that $A_{11}(x_2,x_1) = -A_{11}^\sigma(x_1,x_2)$.
	Both $A_{11}$ and $B_{11}$ are absolutely irreducible.
	%of total degree $5$ and $6$, respectively.
\end{example}

\begin{example}[Polynomials of degree 13]
\label{example:degree-13}
% simplicity : (s,t) = (0,1), p over 11
% normalization : ((-2\beta_{13} + 9)\alpha_{13} + \beta_{13} - 3)/3
	Let $\beta_{13}$
	and $\alpha_{13}$
	be elements of $\QQbar$ satisfying
	%where $\beta_{13}$ satisfies $\beta_{13}^2 - 5\beta_{13} + 3 = 0$,
	\[
		\beta_{13}^2 - 5\beta_{13} + 3 = 0 
		\text{\quad and\quad }
		\alpha_{13}^2 + (\beta_{13}-2)\alpha_{13} + \beta_{13} = 0 .
	\]
	%(so $\alpha_{13}^4 + \alpha_{13}^3 + 2\alpha_{13}^2 - 4\alpha_{13} + 3 = 0$).
	The involution
	$\sigma: \alpha_{13} \mapsto \beta_{13}/\alpha_{13}$ 
	generates $\Gal{\QQ(\alpha_{13})/\QQ(\beta_{13})}$.
	Observe that
	$\QQ(\beta_{13}) = \QQ(\sqrt{13})$ 
	is a real quadratic field,
	and $\QQ(\alpha_{13}) = \QQ(\sqrt{-3\sqrt{13}+1})$ 
	is an imaginary quadratic extension of~$\QQ(\beta_{13})$;
	so $\QQ(\alpha_{13})$ is a CM-field, and
	$\sigma$ acts as complex conjugation.

	Let $f_{13}$ be the polynomial of degree $13$ over $\QQ(\alpha_{13})[t]$
	defined in Table~\ref{table:degree-13-coefficients}:
	\[
		\begin{array}{r@{\;}c@{\;}l}
		f_{13}(x) 
		& = &
		\frac{1}{13} x^{13}
		+
        	((9 \beta_{13} - 39) \alpha_{13} - 6 \beta_{13} + 24) t x^{11}
		\\
		& & \quad \quad \quad {}
		+ ((9 \beta_{13} - 39) \alpha_{13} - 12 \beta_{13} + 51) t x^{10}
		+ \cdots
		\end{array}
	\]
	(note $f_{13}$ is the polynomial of~\cite[\S5.3]{CNC} with $a_1 = \alpha_{13}$).
	\begin{table}
	\caption{Coefficients of the polynomial $f_{13}$ (from Example~\ref{example:degree-13})}
	\label{table:degree-13-coefficients}
	\begin{center}
	\begin{tabular}{|r|l|}
		\hline
		$d$ & Coefficient of $x^d$ in $f_{13}$ \\
		\hline
		\hline
		$ 13 $ & $ 1/13 $ \\
		\hline
		$ 12 $ & $ 0 $ \\
		\hline
		$ 11 $ & $ ((9 \beta_{13} - 39) \alpha_{13} - 6 \beta_{13} + 24) t $ \\
		\hline
		$ 10 $ & $ ((9 \beta_{13} - 39) \alpha_{13} - 12 \beta_{13} + 51) t $ \\
		\hline
		$ 9 $ & $ ((-174 \beta_{13} + 753) \alpha_{13} + (519 \beta_{13} - 2217)) t^2 $ \\
		\hline
		$ 8 $ & $ ((1620 \beta_{13} - 6966) \alpha_{13} - 36 \beta_{13} + 162) t^2 $ \\
		\hline
		$ 7 $ & $ ((-29781 \beta_{13} + 128115) \alpha_{13} + (11988 \beta_{13} - 51651)) t^3 $ \\
		      & ${} + ((1638 \beta_{13} - 7047) \alpha_{13} - 1305 \beta_{13} + 5616) t^2 $ \\
		\hline
		$ 6 $ & $ ((-147933 \beta_{13} + 636498) \alpha_{13} + (135999 \beta_{13} - 585198)) t^3 $ \\
		\hline
		$ 5 $ & $ ((503631 \beta_{13} - 2166939) \alpha_{13} - 585387 \beta_{13} + 2518938) t^4 $ \\
		      & ${} + ((-18036 \beta_{13} + 77598) \alpha_{13} + (119934 \beta_{13} - 516051)) t^3 $ \\
		\hline
		$ 4 $ & $ ((-1130922 \beta_{13} + 4866156) \alpha_{13} - 1672488 \beta_{13} + 7196364) t^4 $ \\
		      & ${} + ((71604 \beta_{13} - 308097) \alpha_{13} - 37719 \beta_{13} + 162297) t^3 $ \\
		\hline
		$ 3 $ & $ ((1827441 \beta_{13} - 7863156) \alpha_{13} + (2618325 \beta_{13} - 11266209)) t^5 $ \\
		      & ${} + ((-8005635 \beta_{13} + 34446465) \alpha_{13} + (3453192 \beta_{13} - 14858316)) t^4 $ \\
		\hline
		$ 2 $ & $ ((50157306 \beta_{13} - 215815671) \alpha_{13} - 31620618 \beta_{13} + 136056429) t^5 $ \\
		      & ${} + ((-3343518 \beta_{13} + 14386410) \alpha_{13} + (3744792 \beta_{13} - 16113006)) t^4 $ \\
		\hline
		$ 1 $ & $ ((-27171504 \beta_{13} + 116912916) \alpha_{13} + (11138796 \beta_{13} - 47927700)) t^6 $ \\
		      & ${} + ((73616121 \beta_{13} - 316753659) \alpha_{13} - 96852267 \beta_{13} + 416733579) t^5 $ \\
		      & ${} + ((770472 \beta_{13} - 3315168) \alpha_{13} - 303912 \beta_{13} + 1307664) t^4 $ \\
		\hline
		$ 0 $ & $ ((-48359916 \beta_{13} + 208081872) \alpha_{13} - 48359916 \beta_{13}) t^6 $ \\
		      & ${} + ((-13260672 \beta_{13} + 57057696) \alpha_{13} - 13260672 \beta_{13}) t^5 $ \\
		\hline
	\end{tabular}
	\end{center}
	\end{table}
	We have a nontrivial factorization
	\(
		f_{13}(x_1) - f_{13}^\sigma(x_2) 
		= 
		A_{13}(x_1,x_2)B_{13}(x_1,x_2)
	\),
	where
	\[
		\begin{array}{r@{\;}l}
		A_{13}(x_1,x_2)
		= 
		&
                x_1^4  + x_2^4
                + (\beta_{13} - 3) x_1^2 x_2^2
                - 9(3 \beta_{13} - 14) t x_1 x_2
                + 12(47 \beta_{13} - 202) t^2
                \\ & {} 
		- ((\beta_{13} - 4) \alpha_{13} + 2) x_1^3 x_2
                + ((\beta_{13} - 4) \alpha_{13} - \beta_{13} + 3) x_1 x_2^3
                \\ & {} 
		+ 3((17 \beta_{13} - 73) \alpha_{13} - 12 \beta_{13} + 50) t x_1^2
                \\ & {} 
		- 3((17 \beta_{13} - 73) \alpha_{13} - 10 \beta_{13} + 45) t x_2^2
                \\ & {} 
		+ 3((5 \beta_{13} - 22) \alpha_{13} - 9 \beta_{13} + 38) t x_1
                \\ & {} 
		- 3((5 \beta_{13} - 22) \alpha_{13} + 2 \beta_{13} - 9) t x_2
		;
		\end{array}
	\]
	note that $A_{13}(x_2,x_1) = A_{13}(x_1,x_2)^\sigma$.
	Both $A_{13}$ and $B_{13}$ are absolutely irreducible.
\end{example}

\begin{example}[Polynomials of degree 15]
\label{example:degree-15}
%\label{subsection:genus-seven}
% simplicity : (s,t) = (0,1), p over 5
% normalization : (-2\alpha + 1)
	Let 
	$\alpha_{15}$ be an element of $\QQbar$ satisfying
	\[
		\alpha_{15}^2 - \alpha_{15} + 4 = 0 .
	\]
	The involution $\sigma : \alpha_{15} \mapsto 4/\alpha_{15}$ 
	generates $\Gal{\QQ(\alpha_{15})/\QQ}$.
	Observe that $\QQ(\alpha_{15}) = \QQ(\sqrt{-15})$ 
	is an imaginary quadratic field,
	and $\sigma$ acts as complex conjugation.

	Let $f_{15}$ be the polynomial of degree $15$ over $\QQ(\alpha_{15})[t]$
	defined in Table~\ref{table:degree-15-coefficients}:
	\[
		f_{15}(x) 
		=
		{\textstyle \frac{1}{15}}x^{15}
		+
        	(\alpha_{15} - 1) t x^{13}
		+
        	(\alpha_{15} + 7) t x^{12}
		+ 
		\cdots
	\]
	(our $f_{15}$ is the polynomial of~\cite[\S5.4]{CNC}
	with $a_1 = \alpha_{15}$.)
	\begin{table}
	\caption{Coefficients of the polynomial $f_{15}$ (from Example~\ref{example:degree-15})}
	\label{table:degree-15-coefficients}
		\begin{center}
		\begin{tabular}{|r|l|r|l|}
		\hline
		$d$  & Coefficient of $x^d$ in $f_{15}(x)$ &
		$d$  & Coefficient of $x^d$ in $f_{15}(x)$ \\
		\hline
		\hline
		$15$ & $ 1/15 $
		&
		$11$ & $ -(5 \alpha_{15} + 21) t^2 $
		\\ 
		\hline
		$14$ & $ 0 $
		&
		$10$ & $ (74 \alpha_{15} - 142) t^2 $
		\\
		\hline
		$13$ & $ (\alpha_{15} - 1) t $
		&
		$9$ & $ -\frac{1}{3} (261 \alpha_{15} - 349) t^3 + (90 \alpha_{15} + 240) t^2 $
		\\
		\hline
		$12$ & $ (\alpha_{15} + 7) t $
		&
		$8$ & $ -(649 \alpha_{15} + 703) t^3 $
		\\
		\hline
		$7$ & \multicolumn{3}{|l|}{$ (138 \alpha_{15} + 717) t^4 + (1380 \alpha_{15} - 5760) t^3 $}\\
		\hline
		$6$ & \multicolumn{3}{|l|}{$ -(2192 \alpha_{15} - 7756) t^4 + (2500 \alpha_{15} + 2800) t^3 $}\\
		\hline
		$5$ & \multicolumn{3}{|l|}{$ \frac{1}{5} (5835 \alpha_{15} - 4743) t^5 - (17790 \alpha_{15} - 5400) t^4 $}\\
		\hline
		$4$ & \multicolumn{3}{|l|}{$ (9699 \alpha_{15} + 6153) t^5 + (300 \alpha_{15} - 74400) t^4 $}\\
		\hline
		$3$ & \multicolumn{3}{|l|}{$ (243 \alpha_{15} - 3591) t^6 + (4680 \alpha_{15} + 92880) t^5 + (21375 \alpha_{15} + 4500) t^4 $}\\
		\hline
		$2$ & \multicolumn{3}{|l|}{$ (7254 \alpha_{15} - 28062) t^6 - (93600 \alpha_{15} - 165600) t^5 $}\\
		\hline
		$1$ & \multicolumn{3}{|l|}{$ -(945 \alpha_{15} + 675) t^7 + (52920 \alpha_{15} - 48600) t^6 - (54000 \alpha_{15} + 216000) t^5 $}\\
		\hline
		$0$ & \multicolumn{3}{|l|}{$ (675 \alpha_{15} - 5400) t^7 - (10800 \alpha_{15} - 86400) t^6 $}\\
		\hline
		\hline
		\end{tabular}
		\end{center}
	\end{table}
	We have a nontrivial factorization
	\(
		f_{15}(x_1) - (-f_{15}^\sigma(x_2)) 
		= 
		A_{15}(x_1,x_2)B_{15}(x_1,x_2)
	\),
	where 
	\[
		\begin{array}{r@{\;}l}
		A_{15}(x_1,x_2) = 
		&
		x_1^7 - (\alpha_{15} - 1) x_1^6 x_2 - 2 x_1^5 x_2^2 + (7 \alpha_{15} - 3) t x_1^5 + (\alpha_{15} + 1) x_1^4 x_2^3 
		\\ & {}
		+ 22 t x_1^4 x_2 + (5 \alpha_{15} + 65) t x_1^4 - (\alpha_{15} - 2) x_1^3 x_2^4 - (10 \alpha_{15} + 2) t x_1^3 x_2^2 
		\\ & {}
		- (50 \alpha_{15} - 70) t x_1^3 x_2 + (9 \alpha_{15} - 69) t^2 x_1^3 - 2 x_1^2 x_2^5 
		\\ & {}
		+ (10 \alpha_{15} - 12) t x_1^2 x_2^3 - 90 t x_1^2 x_2^2 + (39 \alpha_{15} + 33) t^2 x_1^2 x_2 
		\\ & {}
		+ (210 \alpha_{15} - 150) t^2 x_1^2 + \alpha_{15} x_1 x_2^6 + 22 t x_1 x_2^4 + (50 \alpha_{15} + 20) t x_1 x_2^3 
		\\ & {}
		- (39 \alpha_{15} - 72) t^2 x_1 x_2^2 + 450 t^2 x_1 x_2 
		\\ & {}
		- ((63 \alpha_{15} + 45) t^3 - (225 \alpha_{15} + 900) t^2) x_1 + x_2^7 - (7 \alpha_{15} - 4) t x_2^5 
		\\ & {}
		- (5 \alpha_{15} - 70) t x_2^4 - (9 \alpha_{15} + 60) t^2 x_2^3 - (210 \alpha_{15} - 60) t^2 x_2^2 
		\\ & {}
		+ ((63 \alpha_{15} - 108) t^3 
		- (225 \alpha_{15} - 1125) t^2) x_2 - 675 t^3
		;
		\end{array}
	\]
	note that $A_{15}(x_2,x_1) = A_{15}(x_1,x_2)^\sigma$.
	Both $A_{15}$ and $B_{15}$
	are absolutely irreducible.
\end{example}

\begin{example}[Polynomials of degree 21]
\label{example:degree-21}
%\label{subsection:genus-ten}
% simplicity : (s) = (0), p over 5
% normalization : 2
	Let $\alpha_{21}$ be an element of $\QQbar$
	satisfying
	\[
		\alpha_{21}^2 - \alpha_{21} + 2 = 0 .
	\]
	The involution $\sigma: \alpha_{21} \mapsto 2/\alpha_{21}$ 
	generates $\Gal{\QQ(\alpha_{21})/\QQ}$;
	note that $\QQ(\alpha_{21}) = \QQ(\sqrt{-7})$
	is an imaginary quadratic field,
	and $\sigma$ acts as complex conjugation.

	Let $f_{21}$ be the polynomial of degree~$21$ over $k$
	defined in Table~\ref{table:degree-21-coefficients}:
	\[
		f_{21}(x)
		=
		x^{21}
		+
        	(42 \alpha_{21}+ 42) x^{19}
		+
        	(84 \alpha_{21}+ 84) x^{18}
		+ \cdots
	\]
	(note $f_{21}(x) = 2^{21}g(x/2)$, 
	where $g$ is the polynomial of~\cite[\S5.5]{CNC}
	with $a_1 = \alpha_{21}$.)
	\begin{table}
	\caption{Coefficients of the polynomial $f_{21}$ (from Example~\ref{example:degree-21})}
	\label{table:degree-21-coefficients}
		\begin{center}
		\begin{tabular}{|r|l||r|l|}
		\hline
		$d$ & Coefficient of $x^d$ in $f_{21}(x)$ &
		$d$ & Coefficient of $x^d$ in $f_{21}(x)$ \\
		\hline
		\hline
		$21$ & $1$ &
		$20$ & $0$ \\
		\hline
		$19$ & $42 \alpha_{21}+ 42$ &
		$18$ & $84 \alpha_{21}+ 84$ \\
		\hline
		$17$ & $2331 \alpha_{21}- 861$ &
		$16$ & $8820 \alpha_{21}- 2604$ \\
		\hline
		$15$ & $46816 \alpha_{21}- 64568$ &
		$14$ & $227136 \alpha_{21}- 306320$ \\
		\hline
		$13$ & $417060 \alpha_{21}- 1450470$ &
		$12$ & $1249248 \alpha_{21}- 6783504$ \\
		\hline
		$11$ & $-1650124 \alpha_{21}- 18355540$ &
		$10$ & $-25341624 \alpha_{21}- 54772872$ \\
		\hline
		$9$ & $-99408078 \alpha_{21}- 104516426$ &
		$8$ & $-414193752 \alpha_{21}- 32069128$ \\
		\hline
		$7$ & $-1090995696 \alpha_{21}+ 266146344$ &
		$6$ & $-2279293856 \alpha_{21}+ 2006258800$ \\
		\hline
		$5$ & $-4341402044 \alpha_{21}+ 5721876405$ &
		$4$ & $-4332603072 \alpha_{21}+ 10737937392$ \\
		\hline
		$3$ & $-2459323342 \alpha_{21}+ 18242100282$ &
		$2$ & $1708403396 \alpha_{21}+ 16523766868$ \\
		\hline
		$1$ & $8637088971 \alpha_{21}+ 9205492695$ &
		$0$ & $4696767684 \alpha_{21}$ \\
		\hline
		\end{tabular}
		\end{center}
	\end{table}
	We have a nontrivial factorization
	$f_{21}(x_1) - f_{21}^\sigma(x_2) = A_{21}(x_1,x_2)B_{21}(x_1,x_2)$,
	where 
	\[
		\begin{array}{r@{\;}l}
		A_{21}(x_1,x_2) = 
		&
		x_1^5 + (\alpha_{21} + 1) x_1^4 x_2 + 2 \alpha_{21} x_1^3 x_2^2 + (10 \alpha_{21} + 18) x_1^3 
		\\ & {}
		+ (2 \alpha_{21} - 2) x_1^2 x_2^3 + (32 \alpha_{21} - 8) x_1^2 x_2 + (20 \alpha_{21} + 4) x_1^2 
		\\ & {}
		+ (\alpha_{21} - 2) x_1 x_2^4 + (32 \alpha_{21} - 24) x_1 x_2^2 + (32 \alpha_{21} - 16) x_1 x_2 
		\\ & {}
		+ (107 \alpha_{21} + 55) x_1 - x_2^5 + (10 \alpha_{21} - 28) x_2^3 + (20 \alpha_{21} - 24) x_2^2 
		\\ & {}
		+ (107 \alpha_{21} - 162) x_2 + 136 \alpha_{21} - 68
		.
		\end{array}
	\]
	Note that $A_{21}(x_1,x_2) = -A_{21}^\sigma(x_2,x_1)$.
	Both $A_{21}$ and $B_{21}$ are absolutely irreducible.
\end{example}

\begin{example}[Polynomials of degree 31]
\label{example:degree-31}
%\label{subsection:genus-fifteen}
% simplicity : (s) = (0), p over 5
% normalization : 2
	Let $\alpha_{31}$ and $\beta_{31}$ be elements of $\QQbar$
	satisfying
	\[
		\beta_{31}^3 - 13\beta_{31}^2 + 46\beta_{31} - 32 = 0 
		\text{\quad and\quad }
		\alpha_{31}^2 - 1/2(\beta_{31}^2 - 7\beta_{31} + 4)\alpha_{31} + \beta_{31} = 0 
		.
	\]
	The involution $\sigma: \alpha_{31} \mapsto \beta_{31}/\alpha_{31}$
	generates $\Gal{\QQ(\alpha_{31}/\QQ(\beta_{31})}$;
	%(so $
		%\alpha_{31}^6 + \alpha_{31}^5 + 3\alpha_{31}^4 
		%+ 11\alpha_{31}^3 + 44\alpha_{31}^2 + 36\alpha_{31} + 32 
		%= 0
	%$).
	note that $\QQ(\beta_{31})$ is a totally real cubic field,
	and $\QQ(\alpha_{31})$ is a totally imaginary quadratic extension of $\QQ(\beta_{31})$;
	so $\QQ(\alpha_{31})$ is a CM-field,
	and $\sigma$
	acts as complex conjugation.

	Let $f_{31}$ be the polynomial of degree~$31$ over~$k$
	defined in Tables~\ref{table:degree-31-coefficients-1}
	and~\ref{table:degree-31-coefficients-2}:
	\[
		\begin{array}{r@{\;}l}
		f_{31}(x) = 
		&
		\frac{1}{31} x^{31} 
		- (\frac{1}{4} (\beta_{31}^2 - 5 \beta_{31} - 10) \alpha_{31} - (\beta_{31}^2 - 7 \beta_{31} + 12)) x^{29} 
		\\ & {}
		- (\frac{1}{2} (\beta_{31}^2 - 5 \beta_{31} - 10) \alpha_{31} - (2 \beta_{31}^2 - 14 \beta_{31} + 24)) x^{28} 
		+ \cdots
		\end{array}
	\]
	(note $f_{31}(x) = 2^{31}g(x/2)/31$,
	where $g$ is the polynomial of~\cite[\S5.6]{CNC}
	with $a_1 = \alpha_{31}$).
	\begin{table}
	\caption{Coefficients of the polynomial $f_{31}$ (from Example~\ref{example:degree-31}): degrees $14$ through $31$}
	\label{table:degree-31-coefficients-1}
		\begin{center}
		\begin{tabular}{|r|l|}
		\hline
		$d$ & Coefficient of $x^d$ in $f_{31}(x)$ \\
		\hline
		\hline
		$31$ & $1/31$ \\
		\hline
		$30$ & $0$ \\
		\hline
		$29$ & $\scriptstyle \frac{-1}{4} (\beta_{31}^2 - 5 \beta_{31} - 10) \alpha_{31} + \beta_{31}^2 - 7 \beta_{31} + 12$ \\
		\hline
		$28$ & $\scriptstyle \frac{-1}{2} (\beta_{31}^2 - 5 \beta_{31} - 10) \alpha_{31} + 2 \beta_{31}^2 - 14 \beta_{31} + 24$ \\
		\hline
		$27$ & $\scriptstyle \frac{1}{4} (43 \beta_{31}^2 - 1011 \beta_{31} + 2854) \alpha_{31} + \frac{1}{2} (453 \beta_{31}^2 - 3055 \beta_{31} + 3248)$ \\
		\hline
		$26$ & $\scriptstyle (41 \beta_{31}^2 - 977 \beta_{31} + 2802) \alpha_{31} + 886 \beta_{31}^2 - 5986 \beta_{31} + 6496$ \\
		\hline
		$25$ & $\scriptstyle \frac{-1}{4} (17521 \beta_{31}^2 - 74509 \beta_{31} - 60450) \alpha_{31} + 14092 \beta_{31}^2 - 77272 \beta_{31} + 68380$ \\
		\hline
		$24$ & $\scriptstyle \frac{-1}{2} (48519 \beta_{31}^2 - 204491 \beta_{31} - 184718) \alpha_{31} + 80184 \beta_{31}^2 - 442624 \beta_{31} + 403208$ \\
		\hline
		$23$ & $\scriptstyle \frac{-1}{4} (1776161 \beta_{31}^2 - 9373621 \beta_{31} + 3292454) \alpha_{31} 
		+ \frac{1}{2} (2041603 \beta_{31}^2 - 11554557 \beta_{31} + 8612300)$ \\
		\hline
		$22$ & $\scriptstyle -(2942318 \beta_{31}^2 - 15455046 \beta_{31} + 5475220) \alpha_{31} 
		+ 7037348 \beta_{31}^2 - 40203740 \beta_{31} + 30052880$ \\
		\hline
		$21$ & $\scriptstyle \frac{-1}{4} (109481293 \beta_{31}^2 - 596329857 \beta_{31} + 368885054) \alpha_{31} 
		+ 46576255 \beta_{31}^2 - 265263537 \beta_{31} + 187276364$ \\
		\hline
		\multirow{2}{*}{$20$} 
		& $\scriptstyle \frac{-1}{2} (384855193 \beta_{31}^2 - 2112196605 \beta_{31} + 1408837958) \alpha_{31} 
		$ \\ & $ {} 
		\scriptstyle + 307371526 \beta_{31}^2 - 1742220634 \beta_{31} + 1208790968$ \\
		\hline
		\multirow{2}{*}{$19$} 
		& $\scriptstyle \frac{-1}{4} (5290184805 \beta_{31}^2 - 29820077413 \beta_{31} + 21851209042) \alpha_{31} 
		$ \\ & $ {} 
		\scriptstyle + \frac{1}{2} (2521588153 \beta_{31}^2 - 13978683691 \beta_{31} + 9274523664)$ \\
		\hline
		\multirow{2}{*}{$18$}
		& $\scriptstyle -(8697236749 \beta_{31}^2 - 49763738685 \beta_{31} + 38332116082) \alpha_{31} 
		$ \\ & $ {} 
		\scriptstyle + 5911141274 \beta_{31}^2 - 32035079054 \beta_{31} + 20126371040$ \\
		\hline
		\multirow{2}{*}{$17$} 
		& $\scriptstyle \frac{-1}{4} (186111470445 \beta_{31}^2 - 1067698578649 \beta_{31} + 833400031142) \alpha_{31} 
		$ \\ & $ {} 
		\scriptstyle + 9484781350 \beta_{31}^2 - 47546236774 \beta_{31} + 16736919932$ \\
		\hline
		\multirow{2}{*}{$16$} 
		& $\scriptstyle \frac{-1}{2} (494148938071 \beta_{31}^2 - 2839948380571 \beta_{31} + 2256232777618) \alpha_{31} 
		$ \\ & $ {} 
		\scriptstyle - 61154690060 \beta_{31}^2 + 368281842924 \beta_{31} - 366207873944$ \\
		\hline
		\multirow{2}{*}{$15$} 
		& $\scriptstyle \frac{-1}{2} (2214031635615 \beta_{31}^2 - 12716268790027 \beta_{31} + 10156041792602) \alpha_{31} 
		$ \\ & $ {} 
		\scriptstyle - 960101407852 \beta_{31}^2 + 5535136704359 \beta_{31} - 4581193619353$ \\
		\hline
		\multirow{2}{*}{$14$} 
		& $\scriptstyle -(4484463959192 \beta_{31}^2 - 25746958551032 \beta_{31} + 20641481233168) \alpha_{31} 
		$ \\ & $ {} 
		\scriptstyle - 8423937387072 \beta_{31}^2 + 48228838157776 \beta_{31} - 38143305780784$ \\
		\hline
		\end{tabular}
		\end{center}
\end{table}\
\begin{table}
	\caption{Coefficients of the polynomial $f_{31}$ (from Example~\ref{example:degree-31}): degrees $13$ through $0$}
	\label{table:degree-31-coefficients-2}
		\begin{center}
		\begin{tabular}{|r|l|}
		\hline
		$d$ & Coefficient of $x^d$ in $f_{31}(x)$ \\
		\hline
		\hline
		\multirow{2}{*}{$13$} 
		& $\scriptstyle \frac{-1}{4} (63813876335979 \beta_{31}^2 - 367007052549207 \beta_{31} + 296370094708306) \alpha_{31} 
		$ \\ & $ {} 
		\scriptstyle - 52401417590341 \beta_{31}^2 + 299616088960507 \beta_{31} - 233801230247956$ \\
		\hline
		\multirow{2}{*}{$12$}& $
		\scriptstyle \frac{-1}{2} (84595067837587 \beta_{31}^2 - 488413358269471 \beta_{31} + 399412816680130) \alpha_{31} 
		$ \\ & $ {} 
		\scriptstyle - 289909376875898 \beta_{31}^2 + 1656226239390086 \beta_{31} - 1283082623470440$ \\
		\hline
		\multirow{2}{*}{$11$}& $
		\scriptstyle \frac{-1}{4} (276978123366339 \beta_{31}^2 - 1621224937178539 \beta_{31} + 1399602523915382) \alpha_{31} 
		$ \\ & $ {} 
		\scriptstyle - \frac{1}{2} (2756444217062133 \beta_{31}^2 - 15735604159262247 \beta_{31} + 12145338716741672)$ \\
		\hline
		\multirow{2}{*}{$10$}& $
		\scriptstyle (164996225556971 \beta_{31}^2 - 911562557305603 \beta_{31} + 591654846604694) \alpha_{31} 
		$ \\ & $ {} 
		\scriptstyle - 5775442801086222 \beta_{31}^2 + 32951684149353882 \beta_{31} - 25388487691873328$ \\
		\hline
		\multirow{2}{*}{$9$}& $
		\scriptstyle \frac{1}{4} (8153525016709589 \beta_{31}^2 - 46226784686942241 \beta_{31} + 34465661136373590) \alpha_{31} 
		$ \\ & $ {} 
		\scriptstyle - 21765717548444108 \beta_{31}^2 + 124141863300896800 \beta_{31} - 95543137393851316$ \\
		\hline
		\multirow{2}{*}{$8$}& $
		\scriptstyle \frac{1}{2} (21507787300535771 \beta_{31}^2 - 122360462847124879 \beta_{31} + 92829028744745354) \alpha_{31} 
		$ \\ & $ {} 
		\scriptstyle - 71879278651985336 \beta_{31}^2 + 409920655394903344 \beta_{31} - 315310665232998936$ \\
		\hline
		\multirow{2}{*}{$7$}& $
		\scriptstyle \frac{1}{4} (172549107727779319 \beta_{31}^2 - 982848727924637571 \beta_{31} + 750639722104375338) \alpha_{31} 
		$ \\ & $ {} 
		\scriptstyle - \frac{1}{2} (418768591310359209 \beta_{31}^2 - 2388174561757656643 \beta_{31} + 1836495177429186664)$ \\
		\hline
		\multirow{2}{*}{$6$}& $
		\scriptstyle (138365490236826262 \beta_{31}^2 - 788531695474992526 \beta_{31} + 604055823258954628) \alpha_{31} 
		$ \\ & $ {} 
		\scriptstyle - 530716158860110596 \beta_{31}^2 + 3026599060976364972 \beta_{31} - 2327045484274854432$ \\
		\hline
		\multirow{2}{*}{$5$}& $
		\scriptstyle \frac{1}{4} (1461494193805567097 \beta_{31}^2 - 8330939217188411741 \beta_{31} + 6391346186593069190) \alpha_{31} 
		$ \\ & $ {} 
		\scriptstyle - 1132691540565214443 \beta_{31}^2 + 6459518768862357533 \beta_{31} - 4965998974814592772$ \\
		\hline
		\multirow{2}{*}{$4$}& $
		\scriptstyle \frac{1}{2} (1590470411372385357 \beta_{31}^2 - 9067705413825934465 \beta_{31} + 6962808016837221182) \alpha_{31} 
		$ \\ & $ {} 
		\scriptstyle - 1998830101622128910 \beta_{31}^2 + 11398708269008017730 \beta_{31} - 8762745131128427944$ \\
		\hline
		\multirow{2}{*}{$3$}& $
		\scriptstyle \frac{1}{4} (5458654735992646373 \beta_{31}^2 - 31124897594589327589 \beta_{31} + 23912314632422881618) \alpha_{31} 
		$ \\ & $ {} 
		\scriptstyle + \frac{1}{2} (-5512701081507844017 \beta_{31}^2 + 31436273506520022779 \beta_{31} - 24164866978481400776)$ \\
		\hline
		\multirow{2}{*}{$2$}& $
		\scriptstyle (1756872157897042025 \beta_{31}^2 - 10018233805014343961 \beta_{31} + 7698964739179717386) \alpha_{31} 
		$ \\ & $ {} 
		\scriptstyle - 2631501460411936866 \beta_{31}^2 + 15005661005014590390 \beta_{31} - 11533152751494298576$ \\
		\hline
		\multirow{2}{*}{$1$}& $
		\scriptstyle \frac{1}{4} (6099047880687359369 \beta_{31}^2 - 34780055276291665989 \beta_{31} + 26734049819113493038) \alpha_{31} 
		$ \\ & $ {} 
		\scriptstyle - 1489705167473733478 \beta_{31}^2 + 8494536217258921566 \beta_{31} - 6527531886543984036$ \\
		\hline
		\multirow{2}{*}{$0$}& $
		\scriptstyle \frac{1}{2} (1290343630884751523 \beta_{31}^2 - 7358426308111535607 \beta_{31} + 5657092118674073402) \alpha_{31} 
		$ \\ & $ {} 
		\scriptstyle - 2127333184925614050 \beta_{31} + 1673979108081725054$ \\
		\hline
		\end{tabular}
		\end{center}
	\end{table}
	We have a nontrivial factorization
	$f_{31}(x_1) - f_{31}^\sigma(x_2) = A_{31}(x_1,x_2)B_{31}(x_1,x_2)$,
	where
	\[
	\begin{array}{r@{\;}l}
	A_{31}(x_1,x_2) = 
	&
	x_1^{15} + (\frac{1}{4} (\beta_{31}^2 - 9 \beta_{31} + 14) \alpha_{31} - \beta_{31} + 4)x_1^{14}x_2  
	\\ & {} + \cdots \\ & {}
    	+ (\frac{1}{4} (\beta_{31}^2 - 9 \beta_{31} + 14) \alpha_{31} + \frac{1}{2} (\beta_{31}^2 - 7 \beta_{31} + 2))x_1x_2^{14} - x_2^{15} 
	\end{array}
	\]
	is a polynomial of total degree $15$
	satisfying $A_{31}(x_1,x_2) = -A_{31}^\sigma(x_2,x_1)$.
	Both
	$A_{31}$ and $B_{31}$
	are absolutely irreducible.
\end{example}

If $(f,g)$ and $(f',g')$ are equivalent pairs of polynomials,
and $(\family{X},\family{Y})$ and $(\family{X}',\family{Y}')$
the families of pairs of curves associated to $(f,g)$ and $(f',g')$
by the linear or quadratic constructions 
of \S\ref{section:basic-construction},
then $(\family{X},\family{Y})$ and $(\family{X}',\family{Y}')$
are isomorphic.
Indeed,
suppose
$f'(x) = cf(a_1x + b_1) + d$
and $g'(x) = cg(a_2x + b_2) + d$
for 
some $a_1$, $b_1$, $a_2$, $b_2$, $c$, and $d$ in $k$
with $c$, $a_1$ and $a_2$ nonzero.
The isomorphism $(\family{X},\family{Y}) \to (\family{X'},\family{Y'})$
is defined by
\(
	(x_i,y_i) \mapsto (a_ix_1 + b_i, c^{1/2}y_i)
\)
and
\( s \mapsto (s + d)/c \)
for the linear construction,
and by
\( (x_i,y_i) \mapsto (a_ix_1 + b_i, cy_i) \)
and
\( (s_1,s_2) \mapsto ((s_2 + 2d)/c , (s_2 + ds_1 + d^2)/c^2 ) \)
for the quadratic construction.

\begin{lemma}
\label{lemma:family-dimension}
	With the notation of Examples~\ref{example:degree-7} through~\ref{example:degree-31}:
	\begin{enumerate}
		\item	For $n = 11$, $21$, $31$,
			the image of 
			the family
			$\family{X}: y^2 = f_n(x) + s$
			in $\hyperellipticmoduli{(n-1)/2}$
			is one-dimensional;
		\item	For $n = 7$, $13$, $15$,
			the image of 
			the family
			$\family{X}: y^2 = f_n(x) + s$
			in $\hyperellipticmoduli{(n-1)/2}$
			is two-dimensional;
		\item	For $n = 11$, $21$, $31$,
			the image of 
			the family
			$\family{X}: y^2 = f_n(x)^2 + s_1f_n(x) + s_2$
			in $\hyperellipticmoduli{(n-1)}$
			is two-dimensional;
		\item	For $n = 7$, $13$, $15$,
			the image of 
			the family
			$\family{X}: y^2 = f_n(x)^2 + s_1f_n(x) + s_2$
			in $\hyperellipticmoduli{(n-1)}$
			is three-dimensional.
	\end{enumerate}
\end{lemma}
\begin{proof}
	We will show that only finitely many curves in each family $\family{X}$
	can be isomorphic to a given element of $\family{X}$.
	This implies that intersection of $\family{X}(\QQbar)$ 
	with the isomorphism class of a curve $X$
	in $\family{X}(\QQbar)$ 
	is finite,
	and hence that the map from $\family{X}(\QQbar)$ into 
	the moduli space of hyperelliptic curves
	is finite.
	The dimension of the image of $\family{X}$ in the moduli space
	is then equal to the number of parameters of $\family{X}$.

	Consider (1):
	if $X: y^2 = c_0x^n + c_1x^{n-1} + c_2x^{n-2} + \cdots + c_n$
	is a hyperelliptic curve in the family $\family{X}$,
	then the coefficients $c_i$ satisfy conditions
	\[
		\text{(A):\ } c_0 = 1,\quad 
		\text{(B):\ } c_1 = 0,\quad 
		\text{(C):\ } c_2 \not= 0,\quad 
		\text{and}\quad 
		\text{(D):\ } c_3 = \kappa_n c_2,
	\]
	where $\kappa_{11} = 2/\alpha_{11}$,
	$\kappa_{21} = 2$, and $\kappa_{31} = 2$.

	If $X': (y')^2 = f_n(x') + s'$ is isomorphic to $X$,
	then there exists a birational map $\psi: X \to X'$
	defined by
	\[
		\psi: 
		(x,y) \mapsto (x',y') 
		= 
		\left( \frac{\alpha x + \beta}{\gamma x + \delta}, \frac{\epsilon y}{(\gamma x + \delta)^{(n+1)/2}}\right)
	\]
	with $\alpha$, $\beta$, $\gamma$, $\delta$, and $\epsilon$
	in $\QQbar$ sastisfying $\epsilon \not= 0$ and $\alpha\delta - \beta\gamma \not= 0$,
	and $X'$ has a defining equation
	$X': y^2 = \epsilon^{-2}(\gamma x + \delta)^{n+1}(f_n((\alpha x + \beta)/(x + \delta)) + s)$.

	If $\gamma = 0$,
	then we may take $\delta = 1$,
	so
	$\psi(x,y) = (\alpha x + \beta, \epsilon y)$.
	If $X'$ is in $\family{X}$
	then it satisfies (A), (B), (C), and (D).
	Condition (A) implies $\alpha^n = \epsilon^2$,
	while (B) forces $\beta = 0$.
	The coefficients of $x^{n-2}$ and $x^{n-3}$
	in $f_n(\alpha x) + s$
	are then $\alpha^{n-2}c_2$ and $\alpha^{n-3}c_3 = \alpha^{n-3}\kappa_n c_2$,
	whereupon 
	(C) and (D) imply $\alpha = 1$,
	and hence $\epsilon = \pm 1$.
	We conclude that $\psi$ must be either the identity map
	or the hyperelliptic involution, depending on the sign of $\epsilon$;
	in either case, $X' = X$.

	If $\gamma \not= 0$, then we may take $\gamma = 1$.
	For the hyperelliptic polynomial of $X'$ to have degree $n$
	we must have $\delta = -\rho$,
	where $\rho$ is one of the roots of $f_n(x) + s$.
	Conditions (A), (B), (C), and (D) 
	then uniquely determine $\alpha$, $\beta$, $\gamma$, 
	and $\epsilon$ (up to sign)
	in terms of $\rho$ and $\kappa_n$.
	Since there were only $n$ possible choices of $\rho$,
	we find that there are only $2n$ possible choices for $\psi$,
	and only $n$ modulo the hyperelliptic involution.
	
	We have shown that there are only $n+1$ possible defining equations
	for curves in $\family{X}$ isomorphic to $X$
	(in fact, each corresponds to a choice of Weierstrass point of~$X$).
	has a unique defining equation (since the coefficient of $x^0$
	in the defining equation uniquely determines a point of the
	parameter space);
	hence
	there are at most $n+1$ curves in $\family{X}$
	isomorphic to $X$.

	The proof is identical for (2),
	though in this case we must restrict to the open subfamily of $\family{X}$
	where $t \not= 0$ (so that (B) holds),
	and take $\kappa_7 = 1$, $\kappa_{13} = -((2\beta_{13} - 9)\alpha_6 - \beta_{13} + 3)/3$,
	and $\kappa_{15} = -2\alpha_{15} + 1$.
	Again, each curve in $\family{X}$ has a unique defining equation:
	the coefficients of $x^0$ and $x^{n}$ uniquely determine a
	point of the parameter space.

	The proof for (3) and (4)
	is similar, and we only sketch it here.
	This time the curves $X: y^2 = \sum_{i=0}^{2n}c_i x^{2n-i}$ in $\family{X}$ 
	(or the subfamily where $t = 0$ in (4))
	satisfy 
	\[
		\begin{array}{l@{\quad }l@{\quad }l}
			\text{(A'):}\ c_0 = 1,
			&
			\text{(B'):}\ c_1 = 0,
			&
			\text{(C'):}\ c_2 \not= 0,
			\\
			\text{(D'):}\ c_3 = \kappa_n c_2,
			&
			\text{(E'):}\ c_4 \not= 0,
			&
			\text{(F'):}\ c_5 = \lambda_n c_4,
		\end{array}
	\]
	with $\kappa_n$ defined as above
	and 
	\[
		\begin{array}{r@{\;}l}
		\lambda_{7}  = & \frac{1}{277}(44\alpha_{7} + 502) , \\
		\lambda_{11} = & \frac{-1}{1049}(1444\alpha_{11} + 1292) , \\
		\lambda_{13} = & \frac{-1}{24470889}(32177912\beta_{13} - 144562170)\alpha_{13} - 14922610\beta_{13} + 44742102) , \\
		\lambda_{15} = & \frac{-1}{3061}(11624\alpha_{15} - 8242) , \\
		\lambda_{21} = & \frac{-1}{24889}(1872\alpha_21{21} - 98252) ,\ \ \text{and}\\
		\lambda_{31} = & \frac{1}{5572804315201}((-23763234474\beta_{31}^2 + 308913876190\beta_{31} - 904140145396)\alpha_{31} \\
				& {} + (45939160324\beta_{31}^2 - 413033009792\beta_{31} + 22556391264028)) 
		. 
		\end{array}
	\]
	As before,
	the defining equation of any curve in $\family{X}$
	isomorphic to $X$
	is uniquely determined 
	by (A'), (B'), (C'), (D'), (E'), (F'),
	and the choice of a root of $f_n(x)^2 + s_1f_n(x) + s_2$.
	The curves in $\family{X}$ have unique defining equations,
	since
	the coefficients of $x^0$ and $x^n$ (and $x^{2n-2}$ in (4))
	uniquely determine the corresponding point of the parameter space.
	Hence there are at most $2n$ curves in $\family{X}$ isomorphic to~$X$.
\end{proof}

\section{%%%%%%%%%%%%%%%%%%%%%%%%%%%%%%%%%%%%%%%%%%%%%%%%%%%%%%%%%%%%%%%%%%%%%%%
	Explicit Complex and Real Multiplications
}%%%%%%%%%%%%%%%%%%%%%%%%%%%%%%%%%%%%%%%%%%%%%%%%%%%%%%%%%%%%%%%%%%%%%%%%%%%%%%%
\label{section:CM-and-RM}

We now apply the methods of \S\ref{section:basic-construction} 
and~\S\ref{section:kernel-structure}
to the factorizations in Examples~\ref{example:cyclic-polynomials}
and~\ref{example:Dickson-polynomials}.
Most of the resulting families
have already been investigated
elsewhere,
so we treat them only briefly here.
Throughout this section
$n$ denotes an odd prime.

\begin{example}
\label{example:linear-cyclic}
	The linear construction on $(x^n,x^n)$ yields a family
	$(\family{X},\family{X})$
	of pairs of hyperelliptic curves of genus $(n-1)/2$,
	defined by
	$\family{X}: y_i^2 = x_i^n + s$.
	The curves in $\family{X}$ are all isomomorphic
	to the curve $X: y_i^2 = x_i^n + 1$
	(via $(x_i,y_i) \mapsto (\sqrt[n]{s}x_i,\sqrt{s}y_i)$).
	The Jacobian $\Jac{X}$
	is absolutely simple by~\cite[Example 8.4.(1)]{Shimura}.
	The correspondence 
	$C = \variety{y_1 - y_2, x_1 - \zeta_n^ex_2}$
	on $\product[]{X}{X}$
	induces an endomorphism~$\phi$ of $\Jac{X}$.
	Clearly $d(x_1^i)/y_1 = \zeta_n^{ie}d(x_2^i)/y_2$ on $C$,
	so
	\[
		\differentialmatrix[X,Y]{\phi} 
		=
		\mathrm{diag}(\zeta_n^e, \zeta_n^{2e}, \ldots, \zeta_n^{(n-1)e/2})
		.
	\]
	% for $1 \le i \le (n-1)/2$.
	The factors $x_1 - \zeta_n^ex_2$ of $x_1^n - x_2^n$
	therefore correspond to explicit generators for a subring of $\End(\Jac{X})$
	isomorphic to~$\ZZ[\zeta_n]$.
\end{example}

\begin{example}
\label{example:reducible-cyclotomic}
	The quadratic construction 
	on $(x^n, x^n)$
	yields a two-parameter family %$(\family{X},\family{X})$
	$(\family{X}: y_1^2 = x_1^{2n} + s_1 x_1^n + s_2,\ \family{X}: y_2^2 = x_2^{2n} + s_1 x_2^n + s_2)$
	of pairs of hyperelliptic curves of genus $n - 1$.
	Twisting by $(x_i,y_i) \mapsto (s_2^{1/{2n}}x_i,s_2^{1/2}y_i)$,
	we reduce to the one-parameter family
	$\family{X'}: y^2 = x^{2n} + s_1 x^n + 1$
	of~\cite[Remark after Proposition~3]{Tautz-Top-Verberkmoes}.
	The family~$\family{X'}$ has an involution $\iota : (x,y) \mapsto (1/x,y/x^n)$
	which is clearly not the hyperelliptic involution,
	so $\Jacfamily{X'}$ is reducible.
	The correspondences $\variety{y_1 - y_2, x_2 - \zeta_n^ix_1}$
	on $\product{\family{X'}}{\family{X'}}$
	induce
	endomorphisms 
	generating a subring of $\End(\Jacfamily{\family{X}'})$
	isomorphic to $\ZZ[\zeta_n]$,
	as in Example~\ref{example:linear-cyclic}.
	The quotient of $\family{X}'$
	by $\subgrp{\iota}$
	is a one-parameter family of curves of genus $(n-1)/2$
	whose Jacobians have Real Multiplication by $\ZZ[\zeta_n + \zeta_n^{-1}]$.
\end{example}

\begin{example}
\label{example:TTV-curves}
	Let $D_n$ and $A_{n,i}$ be defined as in Example~\ref{example:Dickson-polynomials}.
	The linear construction
	on $(D_n(x),D_n(x))$
	yields
	a one-parameter family 
	\[
		(\family{X}: y_1^2 = D_n(x_1) + s ,\ 
		\family{X}: y_2^2 = D_n(x_2) + s);
	\]
	of pairs of hyperelliptic curves of genus $(n-1)/2$
	over $\QQ$.
	The family $\family{X}$ is identical to 
	the family $\family{C}_t$
	of~\cite[Theorem 1]{Tautz-Top-Verberkmoes}.
	It is shown in~\cite{Tautz-Top-Verberkmoes} that
	the endomorphisms induced by $\variety{y_1-y_2,A_{n,i}(x_1,x_2)}$
	for $1 \le i \le (n-1)/2$
	generate a subring 
	of $\End(\Jacfamily{\family{X}})$
	isomorphic to $\ZZ[\zeta_n + \zeta_n^{-1}]$
	(while $\variety{y_1 - y_2, x_1 - x_2}$
	clearly induces $\multiplication[\Jacfamily{\family{X}}]{1}$).
	It is shown
	that $\Jacfamily{\family{X}}$ is absolutely simple
	for $n > 5$
	in~\cite[Corollary 6]{Tautz-Top-Verberkmoes}
	and for $n = 5$ in~\cite[Remark 15]{Kohel-Smith}.
	The cases $n = 5$ and $n = 7$ of this construction 
	appear as families of efficiently computable endomorphisms
	in~\cite{Kohel-Smith}.
\end{example}

\begin{example}
	Applied to $(D_n(x),D_n(x))$,
	the quadratic construction
	yields a two-parameter family 
	\[
		(\family{X}: y_1^2 = D_n(x_1)^2 + s_1D_n(x_1) + s_2,\ 
		\family{X}: y_2^2 = D_n(x_2)^2 + s_1D_n(x_2) + s_2)
	\]
	of pairs of hyperelliptic curves of genus $n-1$.
	We have a nontrivial factorization
	\[
		\begin{array}{r@{\;}l}
		&
		( D_n(x_1)^2 + s_1D_n(x_1)) - (D_n(x_2)^2 + s_1D_n(x_2)
		\\ = &
		(D_n(x_1) - D_n(x_2))(D_n(x_1) + D_n(x_2) + s_1)
		\\ = &
		\big((x_1 - x_2)\prod_{i=1}^{(n-1)/2} A_{n,i}(x_1,x_2) \big)((D_n(x_1) + D_n(x_2) + s_1))
		.
		\end{array}
	\]
	The correspondences $\variety{y_1-y_2,A_{n,i}(x_1,x_2)}$ on $\product{X}{Y}$
	induce endomorphisms~$\phi_i$ of~$\Jacfamily{X}$
	for $1 \le i \le (n-1)/2$;
	the diagonal correspondence $\variety{y_1-y_2,x_1-x_2}$
	induces $\multiplication[\Jac{X}]{1}$.
	The matrix $\differentialmatrix[\family{X},\family{X}]{\phi_i}$
	is an endomorphism of $\differentials(X)$.
	Since $\differentialmatrix[\family{X},\family{X}]{\phi_i}$ is lower-triangular,
	its characteristic polynomial 
	(and hence that of $\phi_i$) is
	\[
		P(x) = \prod_{j=1}^{(n-1)}(x - t_{j,j}),
	\]
	where $t_{j,j}$ is the $j^\mathrm{th}$ entry on the diagonal of $\differentialmatrix[\family{X},\family{X}]{\phi_i}$:
	that is, $t_{j,j}$ is the leading coefficient of the trace $t_j = \Tr^{\differentials(C_i)}_{\differentials(X)}(d(x_1^j)/y_1)$
	written as a polynomial in $x_2$.
	We~have
	\[
		A_{n,i}(x_1,x_2) 
		= 
		x_1^2 - (\zeta_n^i + \zeta_n^{-i})x_2\cdot x_1 + (x_2^2 + \zeta_n^i - \zeta_n^{-i})
		,
	\]
	so
	\[
		\begin{array}{r@{\;=\;}l}
			t_1 & (\zeta_n^i + \zeta_n^{-i})x_2 
				, \\
			t_2 & (\zeta_n^{2i} + \zeta_n^{-2i})x_2^2 - 2(\zeta_n^{i} - \zeta_n^{-i}) 
				,\ \text{and}\\
			t_j & (\zeta_n^i + \zeta_n^{-i})x_2t_{j-1} - (x_2^2 + \zeta_n^i - \zeta_n^{-i})t_{j-2}
				\ \text{for}\ j > 2;
		\end{array}
	\]
	in particular,
	the coefficients $t_{j,j}$ satisfy
	\( t_{1,1} = \zeta_n^i + \zeta_n^{-i} \),
	\( t_{2,2} = \zeta_n^{2i} + \zeta_n^{-2i} \),
	and
	\[
		t_{j,j} = (\zeta_n^i + \zeta_n^{-i})t_{j-1,j-1} - t_{j-2,j-2} 
		\quad \text{for}\ j > 2 .
	\]
	Solving the second-order linear recurrence, we find
	\( t_{j,j} = \zeta_n^{ij} + \zeta_n^{-ij} \)
	for all $j > 0$,
	so 
	\[
		P(x) = \prod_{j=1}^{(n-1)}(x - (\zeta_n^{ij} + \zeta_n^{-ij}))
		=
		m(x)^2,
	\]
	where \( m \)
	is the minimal polynomial of $\zeta_n + \zeta_n^{-1}$ over~$\QQ$;
	hence $m(\phi_i) = 0$.
	We conclude that 
	$\phi_i$ generates an explicit subring of $\End(\Jacfamily{X})$
	isomorphic to $\ZZ[\zeta_n + \zeta_n^{-1}]$.
\end{example}

\section{%%%%%%%%%%%%%%%%%%%%%%%%%%%%%%%%%%%%%%%%%%%%%%%%%%%%%%%%%%%%%%%%%%%%%%%
	Families of explicit isogenies
}%%%%%%%%%%%%%%%%%%%%%%%%%%%%%%%%%%%%%%%%%%%%%%%%%%%%%%%%%%%%%%%%%%%%%%%%%%%%%%%
\label{sec:examples}

We now apply the methods of \S\ref{section:basic-construction} 
and~\S\ref{section:kernel-structure}
to the factorizations in Examples~\ref{example:degree-7}
through~\ref{example:degree-31}.
The examples in this section
form the proof of Theorem~\ref{theorem:main}.

\begin{example}
\label{example:linear-7}
	Let $f_{7}$, $A_{7}$, $\alpha_{7}$, and $\sigma$ be as in 
	Example~\ref{example:degree-7}.
	The linear construction on
	$(f_{7},f_{7}^\sigma)$
	yields a two-parameter family
	\[
		\left( \family{X}  : y_1^2 = f_{7}(x_1) + s
		,\ 
		\family{Y}  : y_2^2 = f_{7}^\sigma(x_2) + s
		\right)
	\]
	of pairs of hyperelliptic curves of genus~$3$
	defined over $\QQ(\alpha_{7})$.
	Specializing $\family{X}$ at $(s,t) = (1,0)$
	we obtain the curve $X$ of Example~\ref{example:linear-cyclic}
	with $n = 7$,
	where we noted that $\Jac{X}$ was absolutely simple;
	hence the generic fibre of $\Jacfamily{X}$
	is absolutely simple.
	
	The correspondence
	\( \variety{y_1-y_2,A_{7}(x_1,x_2)} \)
	on 
	\( \product{\family{X} }{\family{Y} } \)
	induces a homomorphism
	\( \phi : \Jacfamily{\family{X} } \to \Jacfamily{\family{Y} } \).
	We have
	$$
		\differentialmatrix[X ,Y ]{\phi }
		=
		\left(
		\begin{array}{ccc}
     			\alpha_7 			& 0		& 0	\\
       			0 			& \alpha_7	& 0	\\
			(\alpha_7^\sigma-\alpha_7)t	& 0 		& \alpha_7^\sigma	\\
		\end{array}
		\right)
		,
	$$
	and therefore
	$$
		\differentialmatrix[X ,Y ]{\phi } 
		\differentialmatrix[Y ,X ]{\dualof{\phi }}
		= 
		\differentialmatrix[X ,Y ]{\phi } 
		\differentialmatrix[X ,Y ]{\phi }^\sigma
		=
		2I_{3} ,
	$$
	so $\compose{\dualof{\phi }}{\phi } = \multiplication[\Jacfamily{\family{X} }]{2}$;
	hence
	$\ker\phi \cong \isogenytype{2}{3}$
	by Lemma~\ref{lemma:kernel-structure}.
	The image of $\Jacfamily{X}$ in $\abelianmoduli{3}$
	is two-dimensional by Lemma~\ref{lemma:family-dimension}
	and Torelli's theorem.
	We conclude that $\phi$
	is a two-dimensional family of $\isogenytype{2}{3}$-isogenies
	of (generically) absolutely simple Jacobians,
	thus proving Theorem~\ref{theorem:main}
	for the first row of the table.
\end{example}

\begin{remark}
	More generally,
	given a hyperelliptic curve $X$ of genus~$3$
	and a maximal $2$-Weil isotropic subgroup $S$ of $\Jac{X}[2]$,
	there exists a (possibly reducible) curve $Y$
	of genus~$3$ and
	a $\isogenytype{2}{3}$-isogeny $\phi: \Jac{X} \to \Jac{Y}$
	with kernel~$S$
	(both $Y$ and $\phi$
	may be defined over a quadratic extension of the field of definition of $S$).
	In general, the curve $Y$ is \emph{not} hyperelliptic.
	An algorithm which computes equations for $Y$ and $\phi$
	in the case where $S$ is generated by differences of Weierstrass points
	appears in~\cite{Smith}
	(it is possible to show,
	using techniques similar to those of Lemma~\ref{lemma:two-rank},
	that the kernel of the isogeny 
	of Example~\ref{example:linear-7}
	is not such a subgroup).
	The case where $X$ is non-hyperelliptic
	is treated in~\cite{Lehavi--Ritzenthaler}.
\end{remark}

\begin{example}
\label{example:quadratic-7}
	Let $f_{7}$, $A_{7}$, $\alpha_{7}$, and $\sigma$ be as in
	Examples~\ref{example:degree-7}
	and~\ref{example:linear-7}.
	The quadratic construction on 
	$(f_{7},f_{7}^\sigma)$ yields a three-parameter family
	\[
		\left(
		\family{X}  : 
			y_1^2 = f_{7}(x_1)^2 + s_1f_{7}(x_1) + s_2
		,\ 
		\family{Y}  : 
			y_2^2 = f_{7}^\sigma(x_2)^2 + s_1f_{7}^\sigma(x_2) + s_2
		\right)
	\]
	of pairs of hyperelliptic curves of genus~$6$
	defined over $\QQ(\alpha_{7})$.
	Specializing $\family{X}$ at $(s_1,s_2,t) = (1,0,1)$ 
	and reducing modulo a prime of $\QQ(\alpha_{7})$ over $13$,
	we obtain a curve $\overline{X}$ over $\FF_{13^2}$.
	The Weil polynomial of $\Jac{\overline{X}}$ 
	is irreducible,
	and
	corresponds to the Weil coefficients
	\[
		w_1 = -16,
		w_2 = -46,
		w_3 = 3496,
		w_4 = -36993,
		w_5 = -464728,
		w_6 = 13747140 .
	\]
	Applying Lemma~\ref{lemma:simplicity-criterion},
	we see that $\Jac{\overline{X}}$
	is absolutely simple.
	Hence, the generic fibre of $\Jacfamily{X}$
	is absolutely simple.

	The correspondence
	\( \variety{y_1-y_2,A_{7}(x_1,x_2)} \)
	on
	\( \product{\family{X} }{\family{Y} } \)
	induces a homomorphism
	\( \phi : \Jacfamily{\family{X} } \to \Jacfamily{\family{Y} } \).
	We find that
	$$
		\differentialmatrix[X ,Y ]{\phi}
		=
		\left(\begin{array}{rrrrrr}
		\scriptstyle \alpha_7 & \scriptstyle 0 & \scriptstyle 0 & \scriptstyle 0 & \scriptstyle 0 & \scriptstyle 0 \\
		\scriptstyle 0 & \scriptstyle \alpha_7 & \scriptstyle 0 & \scriptstyle 0 & \scriptstyle 0 & \scriptstyle 0 \\
		\scriptstyle -(2 \alpha_7 + 1) t & \scriptstyle 0 & \scriptstyle \alpha_7^\sigma & \scriptstyle 0 & \scriptstyle 0 & \scriptstyle 0 \\
		\scriptstyle -(\alpha_7 + 4) t & \scriptstyle -2(\alpha_7 + 4) t & \scriptstyle 0 & \scriptstyle \alpha_7 & \scriptstyle 0 & \scriptstyle 0 \\
		\scriptstyle 7(\alpha_7 + 2) t^2 & \scriptstyle -2(2\alpha_7 + 1) t & \scriptstyle 3(\alpha_7^\sigma + 4) t & \scriptstyle 0 & \scriptstyle \alpha_7^\sigma & \scriptstyle 0 \\
		\scriptstyle 7(\alpha_7^\sigma + 4) t^2 & \scriptstyle -7(2\alpha_7 - 3) t^2 & \scriptstyle 3(\alpha_7^\sigma + 4) t & \scriptstyle 4(2\alpha_7^\sigma + 1) t & \scriptstyle 0 & \scriptstyle \alpha_7^\sigma
		\end{array}\right)
		.
	$$
	Since
	$
		\differentialmatrix[Y ,X ]{\dualof{\phi }} 
		= 
		\differentialmatrix[X ,Y ]{\phi }^\sigma
	$,
	we have
	$$
		\differentialmatrix[X ,X ]{\dualof{\phi }\phi }
		=
		\differentialmatrix[X ,Y ]{\phi } 
		\differentialmatrix[Y ,X ]{\dualof{\phi }}
		= 
		2I_{6} ,
	$$
	so $\compose{\dualof{\phi }}{\phi } = \multiplication[\Jacfamily{\family{X} }]{2}$;
	hence $\ker\phi \cong \isogenytype{2}{6}$
	by Lemma~\ref{lemma:kernel-structure}.
	The image of $\Jacfamily{X}$ in $\abelianmoduli{6}$
	is three-dimensional by Lemma~\ref{lemma:family-dimension}
	and Torelli's theorem.
	We conclude that $\phi$ is 
	a three-dimensional family of $\isogenytype{2}{6}$-isogenies
	of (generically) absolutely simple Jacobians,
	thus proving Theorem~\ref{theorem:main}
	for the third row of the table.
\end{example}

\begin{example}
\label{example:linear-11}
% $\kappa = 2/\alpha_{11}$.
	Let $f_{11}$, $A_{11}$, $\alpha_{11}$, and $\sigma$
	be as in Example~\ref{example:degree-11}.
	The linear construction on $(f_{11},f_{11}^\sigma)$,
	yields a one-parameter family 
	\[
		\left(
		\family{X} : y_1^2 = f_{11}(x_1) + s 
		,\ 
		\family{Y} : y_2^2 = f_{11}^\sigma(x_2)  + s
		\right)
	\]
	of pairs of hyperelliptic curves of genus~$5$
	over $\QQ(\alpha_{11})$.
	Specializing $\family{X}$ at $s = 0$ 
	and reducing modulo a prime of $\QQ(\alpha_{11})$ over $7$,
	we obtain a curve $\overline{X}$ over $\FF_{49}$.
	The Weil polynomial of $\Jac{\overline{X}}$ 
	is irreducible, and
	corresponds to the Weil coefficients
	\[
		w_1 = 12,\ 
		w_2 = 28,\ 
		w_3 = -152,\ 
		w_4 = 3652,\ 
		w_5 = 53722\ 
		.
	\]
	Applying Lemma~\ref{lemma:simplicity-criterion},
	we see that $\Jac{\overline{X}}$
	is absolutely simple.
	Hence, the generic fibre of $\Jacfamily{X}$
	is absolutely simple.

	The correspondence
	\( \variety{A_{11}(x_1,x_2),y_1-y_2} \)
	on
	\( \product{\family{X} }{\family{Y} } \)
	induces a homomorphism
	\( \phi : \Jacfamily{\family{X} } \to \Jacfamily{\family{Y} } \).
	We find that
	$$
		\differentialmatrix[X ,Y ]{\phi } 
		= 
		\left(\begin{array}{rrrrr}
		\alpha_{11} & 0 & 0 & 0 & 0 \\
		0 & \alpha_{11}^\sigma & 0 & 0 & 0 \\
		\alpha_{11} + 6 & 0 & \alpha_{11} & 0 & 0 \\
		0 & 0 & 0 & \alpha_{11} & 0 \\
		-3(5\alpha_{11} - 3) & 4(2\alpha_{11} + 1) & 3(\alpha_{11} + 6) & 0 & \alpha_{11}
		\end{array}\right)
		.
	$$
	Since
	$
		\differentialmatrix[Y ,X ]{\dualof{\phi }}
		=	
		\differentialmatrix[X ,Y ]{\phi }^\sigma
	$,
	we have
	$\differentialmatrix[X ]{\dualof{\phi }\phi } = 3I_{5} $,
	so
	$\compose{\dualof{\phi }}{\phi } = \multiplication[\Jacfamily{\family{X} }]{3}$;
	hence $\ker\phi \cong \isogenytype{3}{5}$
	by Lemma~\ref{lemma:kernel-structure}.
	The image of $\Jacfamily{X}$ in $\abelianmoduli{5}$
	is one-dimensional by Lemma~\ref{lemma:family-dimension}
	and Torelli's theorem.
	We conclude that
	$\phi$ is a one-dimensional family of $\isogenytype{3}{5}$-isogenies
	of (generically) absolutely simple Jacobians,
	thus proving Theorem~\ref{theorem:main}
	for the second row of the table.
\end{example}

\begin{example}
\label{example:quadratic-11}
	Let $f_{11}$, $A_{11}$, $\alpha_{11}$, and $\sigma$
	be as in Examples~\ref{example:degree-11}
	and~\ref{example:linear-11}.
	The quadratic construction on $(f_{11},f_{11}^\sigma)$
	yields a two-parameter family
	\[
		\left(
		\family{X} : 
			y_1^2 = f_{11}(x_1)^2 + s_1f_{11}(x_1) + s_2 
		,\ 
		\family{Y} : 
			y_2^2 = f_{11}^\sigma(x_2)^2 + s_1f_{11}(x)  + s_2
		\right) 
	\]
	of pairs of hyperelliptic curves of genus~$10$
	defined over $\QQ(\alpha_{11})$.
	Specializing $\family{X}$ at $(s_1,s_2) = (1,0)$ 
	and reducing modulo a prime of $\QQ(\alpha_{11})$ over $7$,
	we obtain a curve $\overline{X}$ over $\FF_{49}$.
	The Weil polynomial of $\Jac{\overline{X}}$ 
	is irreducible, and
	corresponds to the Weil coefficients
	listed in Table~\ref{table:Weil-coefficients-quadratic-11}.
	\begin{table}
	\caption{Weil polynomial coefficients for Example~\ref{example:quadratic-11}}
	\label{table:Weil-coefficients-quadratic-11}
	\begin{center}
	\begin{tabular}{|r|l|r|l|r|l|r|l|r|l|}
	\hline
	$i$ & $w_i$ &
	$i$ & $w_i$ &
	$i$ & $w_i$ &
	$i$ & $w_i$ &
	$i$ & $w_i$ \\
	\hline
	\hline
	$1$ & $0$ & 
	$2$ & $-16$ &
	$3$ & $196$ &
	$4$ & $2024$ &
	$5$ & $2484$  \\
	$6$ & $35208$ &
	$7$ & $127220$ &
	$8$ & $10074824$ & 
	$9$ & $24089728$ &
	$10$ & $-169499466$ \\
	\hline
	\end{tabular}
	\end{center}
	\end{table}
	Applying
	Lemma~\ref{lemma:simplicity-criterion},
	we see that $\Jac{\overline{X}}$
	is absolutely simple.
	Hence, the generic fibre of $\Jacfamily{X}$
	is absolutely simple.

	The correspondence
	\( \variety{A_{11}(x_1,x_2),y_1-y_2} \)
	on
	\( \product{\family{X} }{\family{Y} } \)
	induces a homomorphism
	\( \phi: \Jacfamily{\family{X} } \to \Jacfamily{\family{Y} } \).
	The $10\!\times\! 10$ matrix
	\( \differentialmatrix[X ,Y ]{\phi } \)
	is lower-triangular,
	with diagonal entries
	\[
		\alpha_{11},\ 
		\alpha_{11}^\sigma,\ 
		\alpha_{11},\ 
		\alpha_{11},\ 
		\alpha_{11},\ 
		\alpha_{11}^\sigma,\ 
		\alpha_{11}^\sigma,\ 
		\alpha_{11}^\sigma,\ 
		\alpha_{11},\ 
		\alpha_{11}^\sigma,
	\]
	each of which is an element of norm $3$
	(we omit the other entries for lack of space).
	We therefore have
	\[
		\differentialmatrix[X ,Y ]{\phi }
		\differentialmatrix[Y ,X ]{\dualof{\phi }}
		=
		\differentialmatrix[X ,Y ]{\phi }
		\differentialmatrix[X ,Y ]{\phi }^\sigma
		= 
		3I_{10},
	\]
	so 
	\(
		\dualof{\phi }\phi  
		= 
		\multiplication[\family{X} ]{3}
	\);	
	hence $\ker \phi \cong \isogenytype{3}{10}$
	by Lemma~\ref{lemma:kernel-structure}.
	The image of $\Jacfamily{X}$ in $\abelianmoduli{10}$
	is two-dimensional by Lemma~\ref{lemma:family-dimension}
	and Torelli's theorem.
	We conclude that $\phi$ is 
	a two-dimensional family
	of $\isogenytype{3}{10}$-isogenies
	of (generically) absolutely simple Jacobians,
	thus proving Theorem~\ref{theorem:main}
	for the sixth row of the table.
\end{example}

\begin{example}
\label{example:linear-13}
	Let $f_{13}$, $A_{13}$, $\alpha_{13}$, $\beta_{13}$ and $\sigma$
	be as in Example~\ref{example:degree-13}.
	The linear construction on $(f_{13},f_{13}^\sigma)$
	yields a two-parameter family 
	\[
		\left(
		\family{X}  : 
		y_1^2 = f_{13}(x_1) + s
		,\ 
		\family{Y}  : 
		y_2^2 = f_{13}^\sigma(x_2) + s
		\right)
	\]
	of pairs of hyperelliptic curves of genus~$6$.
	Specializing $\family{X}$ at $(s,t) = (1,0)$,
	we obtain the curve $X$ of Example~\ref{example:linear-cyclic}
	with $n = 13$,
	which has an absolutely simple Jacobian;
	hence the generic fibre of $\Jacfamily{X}$
	is absolutely simple.

	The correspondence
	\( \variety{A_{13}(x_1,x_2),y_1-y_2} \)
	on
	\( \product{\family{X} }{\family{Y} } \)
	induces a homomorphism
	\( \phi : \Jacfamily{\family{X} } \to \Jacfamily{\family{Y} }  \).
	The $6\!\times\! 6$ matrix
	\(\differentialmatrix[X ,Y ]{\phi }\)
	is lower-triangular;
	if we set $e_1 := (\beta_{13} - 4)\alpha_{13} + 2$
	and $e_2 := \alpha + 1$,
	then the diagonal entries of 
	\( \differentialmatrix[X ,Y ]{\phi } \)
	are
	\[
		e_1,\ \  e_2, \ \ e_1, \ \ e_1^\sigma,\ \ e_2,\ \  e_2 ,
	\]
	each of which is an element of norm $3$ in $\QQ(\beta_{13})$.
	(we omit the other entries for lack of space).
	We therefore have
	\[
		%\differentialmatrix[X ,X ]{\compose{\dualof{\phi }}{\phi }}
		%=
		\differentialmatrix[X ,Y ]{\phi }
		\differentialmatrix[Y ,X ]{\dualof{\phi }}
		=
		\differentialmatrix[X ,Y ]{\phi }
		\differentialmatrix[X ,Y ]{\phi }
		= 
		3I_{6}
		,
	\]
	so $\compose{\dualof{\phi }}{\phi } = \multiplication[\Jacfamily{\family{X} }]{3}$;
	hence $\ker\phi\cong\isogenytype{3}{6}$
	by Lemma~\ref{lemma:kernel-structure}.
	The image of $\Jacfamily{X}$ in $\abelianmoduli{6}$
	is one-dimensional by Lemma~\ref{lemma:family-dimension}
	and Torelli's theorem.
	We conclude that
	$\phi $ is a 
	one-dimensional family of $\isogenytype{3}{6}$-isogenies
	of (generically) absolutely simple Jacobians,
	thus proving Theorem~\ref{theorem:main}
	for the fourth row of the table.
\end{example}

\begin{example}
\label{example:quadratic-13}
	Let $f_{13}$, $A_{13}$, $\alpha_{13}$, $\beta_{13}$ and $\sigma$
	be as in Examples~\ref{example:degree-13}
	and~\ref{example:linear-13}.
	The quadratic construction on $(f_{13}, f_{13}^\sigma)$
	yields
	a three-parameter family 
	\[
		\left(
		\family{X}  : 
		y_1^2 = f_{13}(x_1)^2 + s_1f_{13}(x_1) + s_2
		,\ 
		\family{Y}  : 
		y_2^2 = f_{13}^\sigma(x_2)^2 + s_1f_{13}^\sigma(x_2) + s_2
		\right)
	\]
	of pairs of hyperelliptic curves of genus~$12$
	defined over $\QQ(\alpha_{13})$.
	Specializing $\family{X}$ at $(s_1,s_2,t) = (1,1,1)$ 
	and reducing modulo a prime of $\QQ(\alpha_{13})$ over $5$,
	we obtain a curve $\overline{X}$ over $\FF_{5^4}$.
	The Weil polynomial of $\Jac{\overline{X}}$ 
	is irreducible, and
	corresponds to the Weil coefficients
	listed in Table~\ref{table:Weil-coefficients-quadratic-13}.
	\begin{table}
	\caption{Weil polynomial coefficients for Example~\ref{example:quadratic-13}}
	\label{table:Weil-coefficients-quadratic-13}
	\begin{center}
	\begin{tabular}{|r|l|r|l|r|l|r|l|}
	\hline
	$i$ & $w_i$ & $i$ & $w_i$ &
	$i$ & $w_i$ & $i$ & $w_i$ \\
	\hline
	\hline
	$1$ & $20$ &
	$4$ & $351295$ &
	$7$ & $67298212$ &
	$10$ & $-49877419547660$ \\
	$2$ & $-230$ &
	$5$ & $1293764$ &
	$8$ & $137879604915$ &
	$11$ & $1975333453052116$ \\
	$3$ & $-9232$ &
	$9$ & $-1707055263168$ &
	$6$ & $-204257742$ &
	$12$ & $119629530410659866$ \\
	\hline
	\end{tabular}
	\end{center}
	\end{table}
	Applying Lemma~\ref{lemma:simplicity-criterion},
	we see that $\Jac{\overline{X}}$
	is absolutely simple.
	Hence, the generic fibre of $\Jacfamily{X}$
	is absolutely simple.

	The correspondence
	\( \variety{A_{13}(x_1,x_2),y_1-y_2}) \) on 
	\( \product{\family{X} }{\family{Y} } \)
	induces a homomorphism
	\( \phi : \Jacfamily{\family{X} } \to \Jacfamily{\family{Y} } \).
	The $12\!\times\! 12$ matrix
	\( \differentialmatrix[X ,Y ]{\phi } \)
	is lower-triangular,
	with diagonal entries
	\[
		e_1,\ \ e_2,\ \ e_1,\ \ e_1^\sigma,\ \  e_2,\ \  e_2,\ \  e_2^\sigma,\ \  e_2^\sigma,\ \  e_1,\ \  e_1^\sigma,\ \  e_2^\sigma,\ \  e_1^\sigma 
	\]
	(with $e_1$ and $e_2$ defined as in Example~\ref{example:linear-13}),
	each of which is an element of norm $3$ in $\QQ(\beta_{13})$.
	We therefore have
	\[
		\differentialmatrix[X ,Y ]{\phi }
		\differentialmatrix[Y ,X ]{\dualof{\phi }}
		=
		\differentialmatrix[X ,Y ]{\phi }
		\differentialmatrix[X ,Y ]{\phi }^\sigma
		= 
		3I_{12},
	\]
	so 
	\(
		\dualof{\phi }\phi  
		= 
		\multiplication[\family{X} ]{3}
	\);	
	hence $\ker\phi \cong \isogenytype{3}{12}$
	by Lemma~\ref{lemma:kernel-structure}.
	The image of $\Jacfamily{X}$ in $\abelianmoduli{12}$
	is three-dimensional 
	by Lemma~\ref{lemma:family-dimension}
	and Torelli's theorem.
	We conclude that $\phi$
	is a three-dimensional family of $\isogenytype{3}{12}$-isogenies
	of (generically) absolutely simple Jacobians,
	thus proving Theorem~\ref{theorem:main}
	for the eighth row of the table.
\end{example}

\begin{example}
\label{example:linear-15}
	Let $f_{15}$, $A_{15}$, $\alpha_{15}$, and $\sigma$
	be as in Example~\ref{example:degree-15}.
	The linear construction on~$(f_{15},-f_{15}^\sigma)$
	yields 
	a two-parameter family 
	\[
		\left(
		\family{X} : y_1^2 = f_{15}(x_2) + s 
		,\ 
		\family{Y} : y_2^2 = -f_{15}^\sigma(x_2) + s 
		\right)
	\]
	of pairs of hyperelliptic curves of genus~$7$
	defined over $\QQ(\alpha_{15})$.
	Specializing $\family{X}$ at $(s,t) = (0,1)$ 
	and reducing modulo a prime of $\QQ(\alpha_{15})$ over $17$,
	we obtain a curve $\overline{X}$ over $\FF_{17}$.
	The Weil polynomial $\chi$ of $\Jac{\overline{X}}$
	is irreducible,
	and
	corresponds to the Weil coefficients 
	\[
		w_1 = 0,\  
		w_2 = -4,\ 
		w_3 = -30,\ 
		w_4 = 158,\ 
		w_5 = 972,\ 
		w_6 = -2264,\ 
		w_7 = -18434
		.
	\]
	Applying
	Lemma~\ref{lemma:simplicity-criterion},
	we see that $\Jac{\overline{X}}$
	is absolutely simple.
	Hence, the generic fibre of $\Jacfamily{X}$
	is absolutely simple.

	The correspondence
	\( \variety{A_{15}(x_1,x_2),y_1-y_2} \)
	on
	\( \product{\family{X}}{\family{Y}} \)
	induces a homomorphism
	\( \phi: \Jacfamily{\family{X}}\to\Jacfamily{\family{Y}} \).
	The $7\!\times\! 7$ matrix
	\( \differentialmatrix[X,Y]{\phi} \)
	is lower-triangular
	with diagonal entries 
	\[
		\alpha_{15}^\sigma,\ \alpha_{15}^\sigma,\ -2,\ \alpha_{15}^\sigma,\ 2,\ 2,\ -\alpha_{15},
	\]
	each of which is an element of norm $4$
	(we omit the other entries for lack of space).
	We therefore find 
	\[
		\differentialmatrix[X,Y]{\phi}
		\differentialmatrix[Y,X]{\dualof{\phi}}
		=
		\differentialmatrix[X,Y]{\phi}
		\differentialmatrix[X,Y]{\phi}^\sigma
		=
		4I_{7}
		,
	\]
	so
	$\compose{\dualof{\phi}}{\phi} = \multiplication[\Jacfamily{\family{X}}]{4}$.
	Specializing at $(s,t) = (1,0)$
	and reducing
	modulo a prime %$(2\alpha_{15} - 5)$
	over~$31$,
	we obtain curves 
	$\overline{X}: \bar{y}_1^2 = \bar{x}_1^{15} + 1$ 
	and 
	$\overline{Y}: \bar{y}_2^2 = -\bar{x}_2^{15} + 1$,
	together with an isogeny
	$\overline\phi : \Jac{\overline{X}} \to \Jac{\overline{Y}}$ 
	induced by 
	$
		\variety{\overline{A}(\bar x_1,\bar x_2),\bar y_1-\bar{y}_2}
		\subset
		\product[]{\overline{X}}{\overline{Y}}
	$,
	where
	\[
		\overline{A}
		= \bar{x}_1^7 + 14\bar{x}_1^6\bar{x}_2 - 2\bar{x}_1^5\bar{x}_2^2
		+ 19\bar{x}_1^4\bar{x}_2^3 + 15\bar{x}_1^3\bar{x}_2^4 
		- 2\bar{x}_1^2\bar{x}_2^5 + 18\bar{x}_1\bar{x}_2^6 + \bar{x}_2^7
		.
	\]
	The polynomials $x_1^{15} + 1$ and $-x_2^{15} + 1$
	both split completely over $\FF_{31}$.
	Applying Lemmas~\ref{lemma:two-rank}
	and~\ref{lemma:kernel-from-two-rank},
	we see that 
	$\ker{\overline\phi} \cong \isogenytypetwo{4}{4}{2}{6}$.
	The image of $\Jacfamily{X}$ in~$\abelianmoduli{7}$
	is two-dimensional by Lemma~\ref{lemma:family-dimension}
	and Torelli's theorem.
	We conclude that
	$\phi$ is 
	a two-dimensional family of $\isogenytypetwo{4}{4}{2}{6}$-isogenies
	of (generically) absolutely simple Jacobians,
	thus proving Theorem~\ref{theorem:main}
	for the fifth row of the table.
\end{example}

% FIXME: this goes somewhere...
%	Since $\phi$ is a family of $(\ZZ/4\ZZ)^4\!\!\times\!\!(\ZZ/2\ZZ)^6$-isogenies,
%	there exists a family $\family{A}$ of PPAVs
%	together with families
%	$\phi': \Jacfamily{\family{X}} \to \family{A}$
%	and $\phi'': \family{A} \to \Jacfamily{Y}$
%	of isogenies of PPAVs
%	such that $\ker\phi' \cong \ker\phi'' \cong \isogenytype{2}{7}$
%	and $\phi = \phi''\circ\phi'$.
%	However,
%	there is no reason to expect the PPAVs in $\family{A}$
%	to be Jacobians of curves.

\begin{example}
\label{example:quadratic-15}
	Let $f_{15}$, $A_{15}$, $\alpha_{15}$, and $\sigma$
	be as in Examples~\ref{example:degree-15}
	and~\ref{example:linear-15}.
	The quadratic construction on 
	$(f_{15},-f_{15}^\sigma)$
	yields a three-parameter family
	\[
		\left(
		\family{X}  : 
			y_1^2 = f_{15}(x_2)^2 + s_1f_{15}(x_2) + s_2
		,\ 
		\family{Y}  : 
			y_2^2 = f_{15}^\sigma(x_2)^2 - s_1f_{15}^\sigma(x_2) + s_2 
		\right)
	\]
	of pairs of hyperelliptic curves of genus~$14$
	defined over $\QQ(\alpha_{15})$.
	Specializing at $(s_1,s_2,t) = (1,1,1)$ 
	and reducing modulo a prime of $\QQ(\alpha_{15}$ over $17$,
	we obtain a curve $\overline{X}$ over $\FF_{17}$.
	The Weil polynomial $\chi$ of $\Jac{\overline{X}}$ 
	is irreducible, and
	corresponds to the Weil coefficients
	listed in Table~\ref{table:Weil-coefficients-quadratic-15}.
	\begin{table}
	\caption{Weil polynomial coefficients for Example~\ref{example:quadratic-15}}
	\label{table:Weil-coefficients-quadratic-15}
	\begin{center}
	\begin{tabular}{|r|l|r|l|r|l|r|l|r|l|}
	\hline
	$i$ & $w_i$ & 
	$i$ & $w_i$ &
	$i$ & $w_i$ &
	$i$ & $w_i$ &
	$i$ & $w_i$ \\
	\hline
	\hline
	$1$ & $ -4 $ &
	$4$ & $ -73 $ &
	$7$ & $ 5874 $ &
	$10$ & $ 1252762 $ &
	$13$ & $ 80232390 $ 
	\\
	$2$ & $ 15 $ &
	$5$ & $ 1000 $ &
	$8$ & $ 29004 $ &
	$11$ & $ -1381092 $ &
	$14$ & $ -230738522 $ 
	\\
	$3$ & $ -6 $ &
	$6$ & $ -1182 $ &
	$9$ & $ 22810 $ &
	$12$ & $ 8168424 $ &
	& \\
	\hline
	\end{tabular}
	\end{center}
	\end{table}
	Applying
	Lemma~\ref{lemma:simplicity-criterion},
	we see that $\Jac{\overline{X}}$
	is absolutely simple.
	Hence, the generic fibre of $\Jacfamily{X}$
	is absolutely simple.

	The correspondence
	\( \variety{A_{31}(x_1,x_2),y_1-y_2} \)
	on
	\( \product{\family{X} }{\family{Y} } \)
	induces a homomorphism
	$\phi: \Jacfamily{\family{X} }\to\Jacfamily{\family{Y} }$.
	The $14\!\times\! 14$ matrix
	\( \differentialmatrix[X ,Y ]{\phi } \)
	is lower-triangular
	with diagonal entries 
	\[
		\alpha_{15}^\sigma,\ \alpha_{15}^\sigma,\ -2,\ \alpha_{15}^\sigma,\ 2,\ 2,\ -\alpha_{15},\ \alpha_{15}^\sigma,\ -2,\ -2,\ -\alpha_{15},\ 2,\ -\alpha_{15},\ \alpha_{15},
	\]
	each of which is an element of norm $4$
	(we omit the other entries for lack of space.)
	We therefore find 
	\[
		\differentialmatrix[X ,Y ]{\phi }
		\differentialmatrix[Y ,X ]{\dualof{\phi }}
		=
		\differentialmatrix[X ,Y ]{\phi }
		\differentialmatrix[X ,Y ]{\phi }^\sigma
		=
		4I_{14}
		,
	\]
	so
	$\compose{\dualof{\phi }}{\phi } = \multiplication[\Jacfamily{\family{X} }]{4}$.
	Specializing at $(s_1,s_2,t) = (0,-1,0)$
	and reducing modulo a prime over $31$,
	we obtain curves
	$\overline{X}: \bar{y}_1^2 = \bar{x}_1^{30} - 1$ 
	and 
	$\overline{Y}: \bar{y}_2^2 = \bar{x}_2^{30} - 1$,
	together with an isogeny
	$\overline\phi : \Jac{\overline{X}} \to \Jac{\overline{Y}}$ 
	induced by 
	$
		\variety{\overline{A}(\bar x_1,\bar x_2),\bar y_1-\bar{y}_2}
		\subset
		\product[]{\overline{X}}{\overline{Y}}
	$,
	where
	\[
		\overline{A}
		= \bar{x}_1^7 + 14\bar{x}_1^6\bar{x}_2 - 2\bar{x}_1^5\bar{x}_2^2
		+ 19\bar{x}_1^4\bar{x}_2^3 + 15\bar{x}_1^3\bar{x}_2^4 
		- 2\bar{x}_1^2\bar{x}_2^5 + 18\bar{x}_1\bar{x}_2^6 + \bar{x}_2^7
		.
	\]
	The polynomial $x_i^{30} - 1$
	splits completely over $\FF_{31}$.
	Applying Lemmas~\ref{lemma:two-rank}
	and~\ref{lemma:kernel-from-two-rank},
	we see that 
	$\ker{\overline\phi} \cong \isogenytypetwo{4}{9}{2}{10}$.
	The image of $\Jacfamily{X}$ in $\abelianmoduli{14}$
	is three-dimensional by Lemma~\ref{lemma:family-dimension}
	and Torelli's theorem.
	We conclude that $\phi$ is a family of 
	$\isogenytypetwo{4}{9}{2}{10}$-isogenies
	of (generically) absolutely simple Jacobians,
	thus proving Theorem~\ref{theorem:main}
	for the ninth row of the table.
\end{example}

\begin{example}
\label{example:linear-21}
	Let $f_{21}$, $A_{21}$, $\alpha_{21}$, and $\sigma$
	be as in Example~\ref{example:degree-21}.
	The linear construction on
	$(f_{21},f_{21}^\sigma)$
	yields
	a one-parameter family
	\[
		\left(
		\family{X}  : y_1^2 = f_{21}(x_1) + s 
		,\ 
		\family{Y}  : y_2^2 = f_{21}^\sigma(x_2) + s 
		\right)
	\]
	of pairs of hyperelliptic curves of genus $10$
	over $\QQ(\alpha_{21})$.
	Specializing $\family{X}$ at $s = 0$ 
	and reducing modulo a prime of $\QQ(\alpha_{21})$ over $5$,
	we obtain a curve $\overline{X}$ over $\FF_{25}$.
	The Weil polynomial $\chi$ of $\Jac{\overline{X}}$
	is irreducible,
	and
	corresponds to the Weil coefficients
	listed 
	in Table~\ref{table:Weil-coefficients-linear-21}.
	\begin{table}
	\caption{Weil polynomial coefficients for Example~\ref{example:linear-21}}
	\label{table:Weil-coefficients-linear-21}
	\begin{center}
	\begin{tabular}{|r||l|l|l|l|l|l|l|l|l|l|}
	\hline
	$i$ &
	$1$ & $2$ & $3$ & $4$ & $5$ & $6$ & $7$ & $8$ & $9$ & $10$ \\
	\hline
	$w_i$ & 
	$4$ & $36$ & $272$ & $1268$ & $6492$ & $28540$ & $142200$ & $453284$ & $1065612$ & $17399206$ \\
	\hline
	\end{tabular}
	\end{center}
	\end{table}
	Applying
	Lemma~\ref{lemma:simplicity-criterion},
	we see that $\Jac{\overline{X}}$
	is absolutely simple.
	Hence, the generic fibre of $\Jacfamily{X}$
	is absolutely simple.

	The correspondence
	\( \variety{A_{21}(x_1,x_2),y_1-y_2} \)
	on
	\( \product{\family{X} }{\family{Y} } \)
	induces a homomorphism
	\( \phi  : \Jacfamily{\family{X} } \to \Jacfamily{\family{Y} } \).
	The $10\!\times\! 10$ matrix
        \( \differentialmatrix[X ,Y ]{\phi } \)
	is lower-triangular;
	if we set $e = (\alpha_{21}^\sigma)^2$,
	then the diagonal entries of
        \( \differentialmatrix[X ,Y ]{\phi } \)
	are
	\[
		(\alpha_{21}^\sigma)^2,\ 
		(\alpha_{21}^\sigma)^2,\ 
		-(\alpha_{21}^\sigma)^2,\ 
		(\alpha_{21}^\sigma)^2,\ 
		\alpha_{21}^2,\ 
		-(\alpha_{21}^\sigma)^2,\ 
		\alpha_{21}\alpha_{21}^\sigma,\ 
		(\alpha_{21}^\sigma)^2,\ 
		-\alpha_{21}^2,\ 
		\alpha_{21}^2,
	\]
	each of which is an element of norm $4$
	(we omit the other entries for lack of space).
	We therefore have
	\[       
               	\differentialmatrix[X ,Y ]{\phi }
		\differentialmatrix[Y ,X ]{\dualof{\phi }}
		=
               	\differentialmatrix[X ,Y ]{\phi }
                \differentialmatrix[X ,Y ]{\phi }^\sigma
               	=
		4I_{10}
		,
	\]
	so
	$\compose{\dualof{\phi }}{\phi } = \multiplication[\Jacfamily{\family{X} }]{4}$.
	Specializing at $s = 425$
	and reducing modulo a prime over $599$,
	we obtain curves
	$\overline{X}$
	and
	$\overline{Y}$
	and an isogeny 
	\( \overline{\phi}: \Jac{\overline{X}} \to \Jac{\overline{Y}} \)
	over $\FF_{599}$.
	Applying Lemmas~\ref{lemma:two-rank} and~\ref{lemma:kernel-from-two-rank},
	we find $\ker\overline{\phi}|_{\Jac{\overline{X}}[2]} \cong \isogenytype{2}{11}$,
	so 
	$\ker\phi\cong\isogenytypetwo{4}{9}{2}{2}$.
	The image of $\Jacfamily{X}$ in $\abelianmoduli{10}$
	is one-dimensional by Lemma~\ref{lemma:family-dimension}
	and Torelli's theorem.
	We conclude that
	$\phi $ is 
	a one-dimensional family of 
	$\isogenytypetwo{4}{9}{2}{2}$-isogenies
	of (generically) absolutely simple Jacobians,
	thus proving Theorem~\ref{theorem:main}
	for the seventh row of the table.
\end{example}

\begin{example}
\label{example:quadratic-21}
	Let $f_{21}$, $A_{21}$, $\alpha_{21}$, and $\sigma$
	be as in Examples~\ref{example:degree-21}
	and~\ref{example:linear-21}.
	The quadratic construction on 
	$(f_{21},f_{21}^\sigma)$
	yields
	a two-parameter family
	\[
		\left( 
		\family{X}  : 
		y_1^2 = f_{21}(x_1)^2 + s_1f_{21}(x_1) + s_2 
		,\ 
		\family{Y}  : 
		y_2^2 = f_{21}^\sigma(x_2)^2 + s_1f_{21}^\sigma(x_2) + s_2 
		\right) 
	\]
	of pairs of hyperelliptic curves of genus $10$
	defined over $\QQ(\alpha_{21})$.
	Specializing $\family{X}$ at $(s_1,s_2) = (1,1)$ 
	and reducing modulo a prime of $\QQ(\alpha_{21})$ over $11$,
	we obtain a curve $\overline{X}$ over $\FF_{11}$.
	The Weil polynomial $\chi$ of $\Jac{\overline{X}}$
	is irreducible,
	and
	corresponds to the Weil coefficients listed
	in Table~\ref{table:Weil-coefficients-quadratic-21}.
	\begin{table}
	\caption{Weil polynomial coefficients for Example~\ref{example:quadratic-21}}
	\label{table:Weil-coefficients-quadratic-21}
	\begin{center}
	\begin{tabular}{|r|l|r|l|r|l|r|l|r|l|}
	\hline
	$i$  & $w_i$ &
	$i$  & $w_i$ &
	$i$  & $w_i$ &
	$i$  & $w_i$ &
	$i$  & $w_i$ 
	\\
	\hline
	\hline
	$1$  & $ -4 $ &
	$5$  & $ -1616 $ &
	$9$  & $ -431556 $ &
	$13$ & $ -83783104 $ &
	$17$ & $ -12690445996 $ 
	\\
	$2$  & $ 13 $ &
	$6$  & $ 5919 $ &
	$10$ & $ 1564993 $ &
	$14$ & $ 294134355 $ &
	$18$ & $ 43906230241 $ 
	\\
	$3$  & $ -74 $ &
	$7$  & $ -24382 $ &
	$11$ & $ -5699656 $ &
	$15$ & $ -1000833886 $ &
	$19$ & $ -144999550062 $ 
	\\
	$4$  & $ 403 $ &
	$8$  & $ 105299 $ &
	$12$ & $ 22091457 $ &
	$16$ & $ 3592033583 $ &
	$20$ & $ 476625334323 $ 
	\\
	\hline
	\end{tabular}
	\end{center}
	\end{table}
	Applying Lemma~\ref{lemma:simplicity-criterion},
	we see that $\Jac{\overline{X}}$
	is absolutely simple.
	Hence, the generic fibre of $\Jacfamily{X}$
	is absolutely simple.

	The correspondence
	\( C = \variety{A_{21}(x_1,x_2),y_1-y_2} \)
	on
	\( \product{\family{X}_{20}}{\family{Y}_{20}} \)
	induces a homomorphism
	\( \phi : \Jacfamily{\family{X} } \to \Jacfamily{\family{Y} } \).
	The $20\!\times\! 20$ matrix
        \( \differentialmatrix[X ,Y ]{\phi } \)
	is a lower-triangular;
	if we set $e := -(\alpha_{21} + 1)$,
	then the diagonal entries
	of 
        \( \differentialmatrix[X ,Y ]{\phi } \)
	are
	\[
		\begin{array}{l}
		(\alpha_{21}^\sigma)^2,\ 
		(\alpha_{21}^\sigma)^2,\ 
		-(\alpha_{21}^\sigma)^2,\ 
		(\alpha_{21}^\sigma)^2,\ 
		\alpha_{21}^2,\ 
		-(\alpha_{21}^\sigma)^2,\ 
		\alpha_{21}\alpha_{21}^\sigma,\ 
		(\alpha_{21}^\sigma)^2,\ 
		-\alpha_{21}^2,\ 
		\alpha_{21}^2, \\
		(\alpha_{21}^\sigma)^2,
		-(\alpha_{21}^\sigma)^2,\ 
		\alpha_{21}^2,\ 
		\alpha_{21}\alpha_{21}^\sigma,\ 
		-\alpha_{21}^2,\ 
		(\alpha_{21}^\sigma)^2,\ 
		\alpha_{21}^2,\ 
		-\alpha_{21}^2,\ 
		\alpha_{21}^2,\ 
		\alpha_{21}^2,
		\end{array}
	\]
	each of which is an element of norm $4$
	(we omit the other entries for lack of space).
	We therefore find
	\[       
               	\differentialmatrix[X ,Y ]{\phi }
		\differentialmatrix[Y ,X ]{\dualof{\phi }}
		=
               	\differentialmatrix[X ,Y ]{\phi }
                \differentialmatrix[X ,Y ]{\phi }^\sigma
               	=
		4I_{20}
		,
	\]
	so
	$\compose{\dualof{\phi }}{\phi } = \multiplication[\Jacfamily{\family{X} }]{4}$.
	Specializing at $(s_1,s_2) = (1,6)$
	and reducing at a prime over~$29$,
	we obtain curves
	$\overline{X}$
	and $\overline{Y}$
	and an isogeny
	$\overline{\phi}:\Jac{\overline{X}}\to\Jac{\overline{Y}}$
	over $\FF_{29}$.
	Applying Lemmas~\ref{lemma:two-rank} and~\ref{lemma:kernel-from-two-rank},
	we see that
	$\ker\overline{\phi} \cong \isogenytypetwo{4}{19}{2}{2}$.
	The image of $\Jacfamily{X}$ in~$\abelianmoduli{20}$
	is two-dimensional by Lemma~\ref{lemma:family-dimension}
	and Torelli's theorem.
	We conclude that 
	$\phi$ is 
	a two-dimensional family of $\isogenytypetwo{4}{19}{2}{2}$-isogenies
	of (generically) absolutely simple Jacobians,
	thus proving Theorem~\ref{theorem:main}
	for the eleventh row of the table.
\end{example}

\begin{example}
\label{example:linear-31}
	Let $f_{31}$, $A_{31}$, $\alpha_{31}$, $\beta_{31}$, and $\sigma$
	be as in Example~\ref{example:degree-31}.
	The linear construction on 
	$(f_{31},f_{31}^\sigma)$
	yields 
	a one-parameter family 
	\[
		\left(\family{X}  : y_1^2 = f_{31}(x_1) + s
		,\ 
		\family{Y}  : y_2^2 = f_{31}^\sigma(x_2) + s
		\right)
	\]
	of pairs of hyperelliptic curves of genus~$15$
	over $\QQ(\alpha_{31})$.
	Specializing $\family{X}$ at $s = 0$ 
	and reducing modulo a prime of $\QQ(\alpha_{31})$ over $5$,
	we obtain a curve $\overline{X}$ over $\FF_{5^3}$.
	The Weil polynomial of $\Jac{\overline{X}}$ 
	is irreducible,
	and
	corresponds to the Weil coefficients
	listed in Table~\ref{table:Weil-polynomial-linear-31}.
	\begin{table}
	\caption{Weil polynomial coefficients for Example~\ref{example:linear-31}}
	\label{table:Weil-polynomial-linear-31}
	\begin{center}
	\begin{tabular}{|r|l|r|l|r|l|r|l|}
	\hline
	$i$ & $w_i$ &
	$i$ & $w_i$ &
	$i$ & $w_i$ &
	$i$ & $w_i$ \\
	\hline
	\hline
	$1$ & $25$ &
	$5$ & $146470$ &
	$9$ & $-5019303477$ &
	$13$ & $17625044970092$ \\
	$2$ & $447$ &
	$6$ & $-1950824$ &
	$10$ & $9095279162$ &
	$14$ & $-265293278436450$ \\
	$3$ & $5046$ &
	$7$ & $-61460901$ &
	$11$ & $544453054742$ &
	$15$ & $-4448335615035972$ \\
	$4$ & $42930$ &
	$8$ & $-750851497$ &
	$12$ & $5818130546490$ &
	& \\
	\hline
	\end{tabular}
	\end{center}
	\end{table}
	The Jacobian $\Jac{\overline{X}}$ is absolutely simple
	by Lemma~\ref{lemma:simplicity-criterion};
	hence the generic fibre of $\Jacfamily{X}$
	is absolutely simple.

	The correspondence
	\( C = \variety{A (x_1,x_2),y_1-y_2} \)
	on
	\( \product{\family{X} }{\family{Y} }\ \)
	induces a homomorphism 
	\( \phi : \Jacfamily{\family{X} } \to \Jacfamily{\family{Y} } \).
	The $15\!\times\! 15$ matrix \( \differentialmatrix[X ,Y ]{\phi } \)
	is lower-triangular;
	if we set
	\( e_1 := -((\beta_{31}^2 - 9\beta_{31} + 14)\alpha_{31} + 4\beta_{31} - 16)/4 \),
	\( e_2 := -((\beta_{31} - 6)\alpha_{31} + \beta_{31}^2 - 8\beta_{31} + 8)/2 \),
	and 
	\( e_3 := \alpha_{31} - \beta_{31} + 4 \),
	then the diagonal entries of 
	\( \differentialmatrix[X ,Y ]{\phi } \)
	are
	\[
		e_1,\  
		e_1,\  
		e_2,\  
		e_1,\  
		e_3,\  
		e_2,\  
		e_2^\sigma,\  
		e_1,\  
		e_3,\  
		e_3,\  
		e_3^\sigma,\  
		e_2,\  
		e_3^\sigma,\  
		e_2^\sigma,\  
		e_1^\sigma,
	\]
	each of which is an element of norm $8$ in $\QQ(\beta_{31})$
	(we omit the other entries for lack of space).
	We therefore find
	\[
		\differentialmatrix[X ,Y ]{\phi }
		\differentialmatrix[Y ,X ]{\dualof{\phi }} 
		= 
		\differentialmatrix[X ,Y ]{\phi }
		\differentialmatrix[X ,Y ]{\phi }^\sigma
		=
		8I_{15}
		,
	\]
	so
	$\compose{\dualof{\phi }}{\phi } = \multiplication[\Jacfamily{\family{X} }]{8}$.
	Specializing at $s = 0$
	and reducing modulo a prime over~$47$,
	we obtain curves 
	$\overline{X}$
	and
	$\overline{Y}$
	and an isogeny $\overline{\phi}: \Jac{\overline{X}}\to\Jac{\overline{Y}}$
	over $\FF_{47}$.
	Applying Lemmas~\ref{lemma:two-rank} 
	and~\ref{lemma:kernel-from-two-rank},
	we find $\ker\overline{\phi} \cong \isogenytypethree{8}{5}{4}{10}{2}{10}$.
	The image of $\Jacfamily{X}$ in $\abelianmoduli{15}$
	is one-dimensional by Lemma~\ref{lemma:family-dimension}
	and Torelli's theorem.
	We conclude that $\phi$
	is a one-dimensional family of 
	$\isogenytypethree{8}{5}{4}{10}{2}{10}$-isogenies
	of (generically) absolutely simple Jacobians,
	thus proving Theorem~\ref{theorem:main}
	for the tenth row of the table.
\end{example}

\begin{example}
\label{example:quadratic-31}
	Let $f_{31}$, $A_{31}$, $\alpha_{31}$, $\beta_{31}$, and $\sigma$
	be as in Examples~\ref{example:degree-31}
	and~\ref{example:linear-31}.
	The quadratic construction on 
	$(f_{31},f_{31}^\sigma)$
	yields a 
	two-parameter family 
	\[
		\left( 
		\family{X}  : 
		y_1^2 = f_{31}(x_1)^2 + s_1f_{31}(x_1) + s_2
		,\ 
		\family{Y}  : 
		y_2^2 = f_{31}^\sigma(x_1)^2 + s_1f_{31}^\sigma(x_2) + s_2
		\right)
	\]
	of pairs of hyperelliptic curves of genus~$30$
	defined over $\QQ(\alpha_{31})$.
	Specializing $\family{X}$ at $(s_1,s_2) = (1,2)$ 
	and reducing modulo a prime of $\QQ(\alpha_{31})$ over $3$,
	we obtain a curve~$\overline{X}$ over $\FF_{3^6}$.
	The Weil polynomial of $\Jac{\overline{X}}$ 
	is irreducible,
	and
	corresponds to the Weil coefficients
	listed in Table~\ref{table:Weil-coefficients-quadratic-31}.
	\begin{table}
	\caption{Weil polynomial coefficients for Example~\ref{example:quadratic-31}}
	\label{table:Weil-coefficients-quadratic-31}
	\begin{center}
	\begin{tabular}{|r|l|r|l|}
	\hline
	$i$  & $w_i$ &
	$i$  & $w_i$ \\
	\hline
	\hline
	$1$  & $86$ &
	$14$ & $2538874803438283085247$ \\
	$2$  & $3451$ &
	$15$ & $75551657032201511555544$ \\
	$3$  & $87828$ &
	$16$ & $2132291122470015060842077$ \\
	$4$  & $1643613$ &
	$17$ & $58726738607409603792625818$ \\
	$5$  & $43045482$ &
	$18$ & $1634122583940469502202897151$ \\
	$6$  & $1781887735$ &
	$19$ & $48450321094461320825161410124$ \\
	$7$  & $76936315232$ &
	$20$ & $1504867060985705824450391696293$ \\
	$8$  & $3105710470069$ &
	$21$ & $45345655631250765718117003095430$ \\
	$9$  & $102095895729754$ &
	$22$ & $1270533776275133738442812562176203$ \\
	$10$ & $2779643454835731$ &
	$23$ & $34526697723237826449755783511899672$ \\
	$11$ & $71233879362094240$ &
	$24$ & $956449237011673888073922827627521777$ \\
	$12$ & $2193677250388156081$ &
	$25$ & $26767220948731629452685495358053131182$ \\
	$13$ & $77619720346267760370$ & 
	$26$ & $757441695740127512275452904130818491239$ \\ 
	\hline
	$27$ & \multicolumn{3}{|l|}{$21123226183916202851140834209673472022292$} \\
	$28$ & \multicolumn{3}{|l|}{$575803060349811307421020344590821665754597$} \\
	$29$ & \multicolumn{3}{|l|}{$15365239367923178818677513358710798508553810$} \\
	$30$ & \multicolumn{3}{|l|}{$408015365744689122150660893862413952306834751$} \\
	\hline
	\end{tabular}
	\end{center}
	\end{table}
	The Jacobian $\Jac{\overline{X}}$ is absolutely simple
	by Lemma~\ref{lemma:simplicity-criterion};
	hence the generic fibre of $\Jacfamily{X}$
	is absolutely simple.

	The correspondence
	\( C = \variety{A_{31}(x_1,x_2),y_1-y_2} \)
	on
	\( \product{\family{X} }{\family{Y} } \)
	induces a homomorphism
	\( \phi : \Jacfamily{\family{X} } \to \Jacfamily{\family{Y} } \).
	The $30\!\times\! 30$ matrix \( \differentialmatrix[X ,Y ]{\phi } \);
	is lower-triangular
	with diagonal entries 
	\[
		\begin{array}{l}
		e_1,\  e_1,\  e_2,\  e_1,\  e_3,\  e_2,\  e_2^\sigma,\  e_1,\  e_3,\  e_3,\  
		e_3^\sigma,\  e_2,\  e_3^\sigma,\  e_2^\sigma,\  e_1^\sigma,\ 
		\\
		e_1,\  e_2,\  e_3,\  e_2^\sigma,\  e_3,\  e_3^\sigma,\  e_3^\sigma,\  
		e_1^\sigma,\  e_2,\  e_2^\sigma,\  e_3^\sigma,\  e_1^\sigma,\  
		e_2^\sigma,\  e_1^\sigma,\  e_1^\sigma
		\end{array}
	\]
	(where $e_1$, $e_2$, and $e_3$ are defined as in Example~\ref{example:linear-31}),
	each of which is an element of norm $8$ in $\QQ(\beta_{31})$
	(we omit the other entries for lack of space).
	We therefore have
	\[
		\differentialmatrix[X ,Y ]{\phi }
		\differentialmatrix[Y ,X ]{\dualof{\phi }} 
		= 
		\differentialmatrix[X ,Y ]{\phi }
		\differentialmatrix[X ,Y ]{\phi }^\sigma
		=
		8I_{30}
		,
	\]
	so
	$\compose{\dualof{\phi }}{\phi } = \multiplication[\Jacfamily{\family{X} }]{8}$.
	Specializing at $(s_1,s_2) = (4,9)$
	and reducing modulo a prime over $47$,
	we obtain curves $\overline{X}$ and $\overline{Y}$
	and an isogeny $\overline{\phi}:\Jac{\overline{X}} \to \Jac{\overline{Y}}$
	over $\FF_{47}$.
	Applying Lemmas~\ref{lemma:two-rank} and~\ref{lemma:kernel-from-two-rank},
	we find
	$\ker\overline{\phi} \cong \isogenytypethree{8}{11}{4}{19}{2}{19}$.
	The image of $\Jacfamily{X}$ in $\abelianmoduli{30}$
	is two-dimensional by Lemma~\ref{lemma:family-dimension}
	and Torelli's theorem.
	We conclude that $\phi$
	is a two-dimensional family of 
	$\isogenytypethree{8}{11}{4}{19}{2}{19}$-isogenies
	of (generically) absolutely simple Jacobians,
	thus proving Theorem~\ref{theorem:main}
	for the twelfth row of the table.
\end{example}

\end{document}